\documentclass[final, journal]{IEEEtran}
\usepackage{amsmath}
\usepackage{amsmath,amssymb,amsbsy,amsfonts, bm,mathrsfs,amsthm,amscd}
\usepackage{graphicx,caption, subcaption}
\usepackage[mediumspace,mediumqspace,Grey,squaren]{SIunits}
\usepackage{pstool}
\usepackage{hyperref}
\usepackage{cite, array}

\newtheorem{claim}{Claim}[section]
\newtheorem{lemma}[claim]{Lemma}

\newtheorem{theorem}{Theorem}

\newtheorem{prefact}{{\bf Fact}}[section]

\newcommand{\comment}[1]{}

\def\bR{\mathbb{R}}
\def\bP{\mathbb{P}}
\def\bcE{\bar{\mathcal{E}}}
\def\hcE{\widehat{\mathcal{E}}}
\def\ut{\underline{t}}
\def\ot{\overline{t}}
\def\ui{\underline{i}}
\def\oi{\overline{i}}
\def\he{\widehat{e}}
\def\hx{\widehat{x}}
\def\boldx{\boldsymbol{x}}
\def\boldy{\boldsymbol{y}}
\def\boldw{\boldsymbol{w}}

\def\boldtau{\boldsymbol{\tau}}
\def\boldtw{\widetilde{\boldsymbol{w}}}
\def\tell{\widetilde{\ell}}

\def\expect{\mathbb{E}}
\def\reals{\mathbb{R}}

\def\argmin{\rm{argmin}}  
\def\bessel{\mathrm{I}}

\def\fnr{\rm FNR }
\def\fdr{\rm FDR }
\def\tpr{\rm TPR }

\def\newtof{{\rm ATOF}}
\def\ctof{\rm{TOF}}
\begin{document}

\title{Accelerated Time-of-Flight Mass Spectrometry}
\author{Morteza~Ibrahimi,~\IEEEmembership{Student Member,~IEEE,}
        Andrea~Montanari,~\IEEEmembership{Senior Member, IEEE,}\\
        and~George~Moore,~\IEEEmembership{Senior Member, IEEE}
\thanks{M. Ibrahimi is with the Department of Electrical Engineering, Stanford University, Stanford,
CA, 94305 USA email: ibrahimi@stanford.edu.}
\thanks{A. Montanari is with the Department of Electrical Engineering and Department of Statistics, Stanford University, Stanford, CA, 94305 USA email: montanari@stanford.edu.}
}

%

\date{} 
\maketitle

\begin{abstract}
We study a simple modification to the conventional time of flight mass spectrometry (TOFMS) where a \emph{variable} and (pseudo)-\emph{random} pulsing rate is used which allows for traces from
different pulses to overlap.
This modification requires little alteration to the currently employed hardware.
However, it requires a reconstruction method to recover the spectrum from
highly aliased traces.
We propose and demonstrate an efficient algorithm that can process massive TOFMS data using computational resources that can be considered modest with today's standards.
This approach can be used to improve duty cycle, speed, and mass
resolving power of TOFMS at the same time. We expect this to extend the applicability 
of TOFMS to new domains.
\end{abstract}


\section{Introduction}

Mass spectrometry (MS) refers to a family of techniques used to analyze the
constituent chemical species in a sample.
The applications abound in science
and technology and new fields of scientific investigations have evolved around
these techniques.
An example is proteomics which refers to the science of analyzing peptides and proteins.
Proteins are the workhorse of many biological mechanism.  
Of great interest to biological sciences, medical research and drug discovery
and developments is identifying and analyzing the composition and structure of
proteins and other large chemical compounds.
It has become possible only recently to analyze the composition of proteins
with high throughput and accuracy through mass spectrometry techniques \cite{hillenkamp1991matrix} \cite{fenn1989electrospray}.
Other applications include measuring isotopic ratio, space exploration, testing
for illegal substances etc.
Mass spectrometers are usually accompanied with gas or liquid
chromatography techniques and are used in different configurations in tandem with other
mass spectrometer of the same or different types.
These configurations provide a wide range of utility and performance criteria making mass spectrometry relevant for many different applications.

A typical mass spectrometer consists of three main modules: an ionizer, a mass
analyzer, and a detector. 
The ionizer converts the species of interest, and possibly other compounds in the sample to ions in gas phase. 
Recent advances in ionization techniques, namely matrix assisted laser desorption ionization \cite{doi:10.1021/ac00024a002}, and electrospray ionization \cite{fenn1989electrospray} has made it possible to ionize and transform into gas phase large intact molecules like proteins. 
These techniques provided new applications for mass spectrometry and open new avenues for analyzing the composition and structure of proteins \cite{ragoussis2006matrix}.

The purpose of the mass analyzer module is to separate the ions according to their mass to charge ratio (MCR). 
Today's common mass analyzers separate the ions by subjecting them to electromagnetic fields.
These fields exert different forces to different ions.
One class of instruments, broadly referred to as sector instruments cause the ions with different MCR to take different trajectories, effectively beamforming a stream of flying ions of particular MCR toward a detector \cite{cross1951two}
Another technique is to have all the ions travel a common trajectory but with different velocities.
This is the basis for time of flight (TOF) mass spectrometers which we shall describe throughly in the sequel.
A third approach is to \textit{guide} only a particular MCR to have a stable trajectory. 
This is the basis for the Quadrupole and ion trap mass analyzers \cite{schwartz2002two}.
These instruments can act as a MCR filter or scan a wider range by sweeping the filter pass band.

The detector module senses the ions by detecting the impact of charged compounds with the detector surface or the charge or current they induce by their particular motion.

\begin{figure}\center
\includegraphics[width=.5\textwidth]{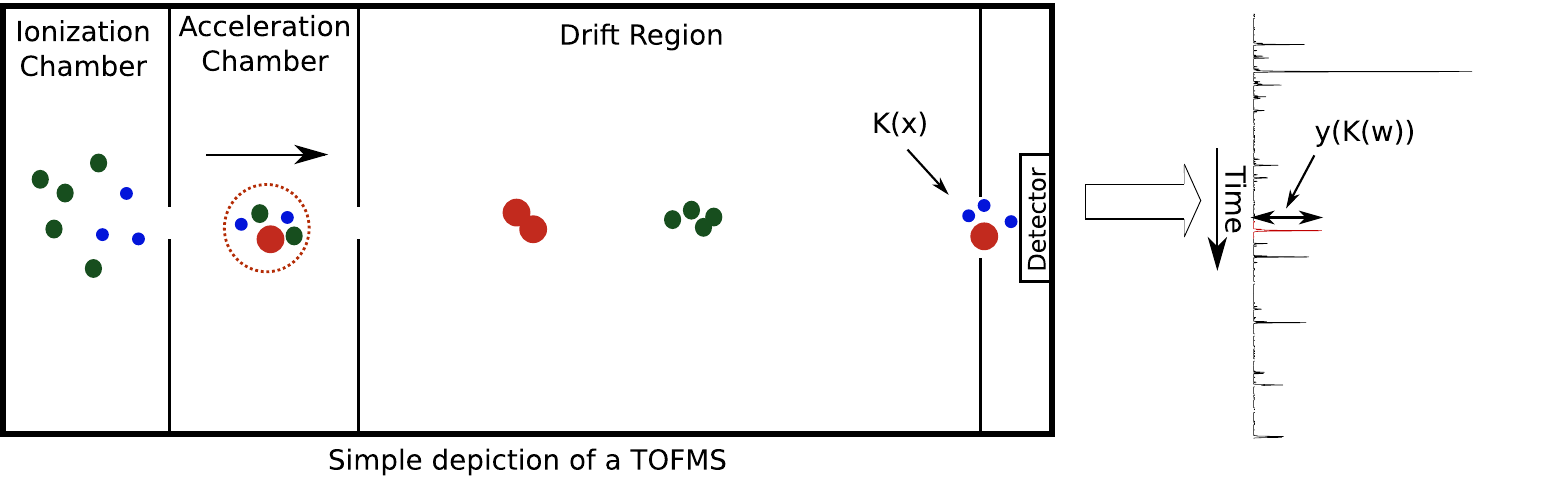}
\caption{Different parts of a TOFMS.}
\label{fig:TOFMS_schematic}
\end{figure}
In this paper we are concerned with time of flight mass spectrometry (TOFMS) which is a simple yet powerful MS technique.
TOFMS was introduced in the 1940s by Stephens
\cite{stephens1946pulsed}.
TOFMS offers two major benefits over alternative techniques.  It has essentially
unlimited mass range and high repetition rate.  These properties along with the
recent advances in available hardware and ionization techniques have made TOFMS
an appealing choice for the analysis of samples with wide mass range
\cite{roberts2004application}, biological macromolecules
\cite{fenn2003electrospray} and in combination with other mass spectrometers
\cite{shevchenko1997rapid}.

A basic TOF mass spectrometer consists of four parts: an ionization chamber, an acceleration chamber, a drift region, and a detector (c.f., Fig~\ref{fig:TOFMS_schematic}).
The sample is ionized in the ionization chamber.
These ions are then subjected to a very strong electrical field
in the acceleration chamber, effectively firing them into the drift region.
Ideally, the ions entering the drift region have kinetic
energy, $K$, proportional to their charge $z$, i.e., if the potential difference in the acceleration chamber is $U$ then the following holds $K = Uz$.
This means that an ion with mass $m$ has velocity $v = \sqrt{\frac{2Uz}{m}}$. 
Therefore, assuming that the length of the drift region to be $L$, the time to reach the detector is
\begin{equation}
t = \sqrt{\frac{m}{2z}}\frac{L}{2U}.
\end{equation}
In other words, the time that takes for an ion to reach the detector is proportional to $\sqrt{m/z}$ where $m$ is the mass of the
ion and $z$ is its net charge.

TOFMS is a pulsed technique, i.e., ions are formed in an ionization stage and
subsequently accelerated as a \textit{packet} into the drift region with ions with different MCR traveling at different speeds.
As the ions impact the detector, they generate a continuous electrical signal which is then sampled resulting in a discrete signal.  
The result of this process, which we call a \textit{scan}, is a noisy sample of the $\sqrt{m/z}$ spectrum.  
A single scan is often too noisy and this process is repeated from hundreds to a few thousands times and averaged to obtain an accurate estimate of the $\sqrt{m/z}$ spectrum.
We will call an estimate of $\sqrt{m/z}$, which is the result of processing many scans, one \textit{acquisition}.
In many applications multiple consecutive acquisitions are obtained: to
construct a movie of an evolving sample like an ongoing chemical reaction; to analyze the output of a preceding chromatography stage; or to mass analyze samples through an automated system where TOFMS instrument in being fed automatically, e.g., in pharmaceutical applications \cite{janiszewski2008perspectives}.  

There are several metrics that describe the performance of a mass spectrometer.
Some of the widely used metrics are: mass resolving power, mass accuracy, dynamic range,
sensitivity and speed.
Mass resolving power is the minimum difference in mass to charge ratio for two present species to be distinguishable by the instrument.
Mass accuracy is the normalized precision with which the instrument can report
the MCR of present species, measured in MCR error divided by MCR.
Mass range refers to the range of MCRs the instrument can detect. 
Sensitivity is the minimum concentration of an specie to be detectable by the instruments.
Finally, speed is the number of acquisition the instrument can acquire
per unit time.

These metrics are not independent.
Several trade offs exist among these
metrics based on how a TOFMS instrument is designed and operated.
For example, speed can be increased, at the cost of mass resolving power,
mass accuracy and sensitivity, by decreasing the number of scans collected for
each acquisition.
Another trade off exists between speed in one hand and mass resolving power and
accuracy on the other hand, which is the focus of this paper and is described in
details below.

In conventional TOFMS, the time between consecutive pulses is set to be long enough to avoid overlap between different scans, i.e., for the slowest ion in an scan to arrive at the detector before the fastest ion of the next scan.
Hence, acquisition time is lower bounded by,
\begin{align}\label{eq:acquisition_time_lbdd}
T_{\text{acquisition}} &\ge N \times T^*_{\text{scan}}  \nonumber \\
& \ge N \times \frac{L}{\sqrt{2U}} \,(\sqrt{(m/z)_{\max}} - \sqrt{(m/z)_{\min}}),
\end{align}
where $N$ is the number of scans collected for each acquisition.
Furthermore, the difference in time of arrival for two ions is 
\begin{equation}
  t_2 - t_1 = \frac{L}{\sqrt{2U}} \, (\sqrt{m_2/z_2} - \sqrt{m_1/z_1}).
\end{equation}

Hence, increasing the length of the drift region $L$ $(i)$ increases the
acquisition time and therefore decreases the speed $(ii)$ spreads the ions further apart and therefore increases the mass accuracy and resoling power. 
This issue is of fundamental importance because of the following. 

Other factors that can improve the mass accuracy and resolving power of
a TOFMS, e.g., detector characteristics and the speed of the electronics, have
reached their limits while new applications demand even better performance in
terms of mass accuracy and resolving power.
One option that remains available for improving the mass accuracy and resolving
power is to increase the length of the drift region.

However, first, there is the obvious desire for higher speed and accuracy at the same time.
Second, some applications have stringent requirements in terms of speed, mass
accuracy, resolving power and sensitivity.
This could be due to exogenous time restrictions, e.g., when monitoring a
chemical reaction or experimentation choice, e.g., when TOFMS is preceded by a
chromatography stage or used in tandem,\comment{\cite{boyd1994linked}} with
another mass spectrometry stage.
There is also significant economical implications, a high end instrument costs
at the order of hundreds of thousands of dollars and improving the speed and
throughput while keeping or improving the accuracy can result in significant
savings.
This is most clear in the case of large scale automated experiments used in drug
discovery and development activities.

Therefore, simultaneous improvement of the speed and mass accuracy and resolving power is of fundamental interest in TOFMS \cite{trapp2004continuous}.

Conventional TOFMS works by repeating the same experiment multiple times and
averaging the results.
The choice of averaging was mainly due to its simplicity.

In particular, the volume and rate of data generated by TOFMS instruments prohibited the use of more sophisticated techniques.
Our ability to commit more computational resources has increased significantly since the introduction of TOFMS.
At the same time, the data rate of these instruments generate has also dramatically increased.
In this paper, we present an efficient, highly parallelizable algorithm that in
conjunction with a simple modification to the conventional TOFMS can improve
mass accuracy, mass resolving power and speed at the same time.

There has been previous work trying to alleviate this problem.  One approach,
called Fourier transform TOF, is to modulate a continuous ion beam at the source
using a periodic waveform and subsequently accelerate it into the drift region
\cite{knorr1986fourier}.  The detected signal is then demodulated to obtain the
spectrum.  Another approach, called Hadamard transform TOF (HT-TOF)
\cite{brock1998hadamard}, is based on modulation (gating) of a continuous ion
source by a $0/1$ pulse. 
In this approach, an ion beam is deflected according to a pseudorandom sequence of pulses. 
If the pulse is $1$, the beam is undeflected and will reach the detector. 
In contrast, if the pulse is $0$ the beam is deflected away from the detector. 
The pulse sequence has the same frequency as the detector.
In an ideal case where there is neither shot noise nor additive noise, the output can be described as 
\begin{equation}\label{eq:HTTOF_input_output} 
\boldy = H \boldx,
\end{equation} 
where $\boldy \in \reals^n$ is the observed signal at the detector and $\boldx \in \reals^n$ is the TOF spectrum.
$H \in \reals^{n \times n}$ is a $0/1$ matrix where each column is the pseudorandom sequence shifted by the index of the column.
As long as $H$ is full rank and the model is accurate the TOF spectrum can be recovered by applying the inverse of $H$ to $\boldy$. 
The spectrum is obtained by a deconvolution that can be implemented efficiently using the fast Hadamard transform.
 
One drawback of these methods is that they treat the reconstruction process as a
deterministic inversion problem and ignore the noisy nature of the observations.
Furthermore, they require substantial modification to the hardware of a
conventional TOFMS.
In this paper,  we describe a different method, called
accelerated time of flight mass spectrometry (\newtof{}), which simultaneously achieves mass resolving power, duty
cycle, and speed improvement using essentially the same hardware as a
conventional TOFMS.  Our reconstruction scheme acknowledges the stochastic
nature of the observation.  Simulation results using real data confirm the
performance improvement of this scheme.

\noindent{\bf Notations and Terminology:}\label{sec:problem_formulation} 
Let $[n] = \{1, 2, \dots, n\}$ and $\{\boldx[i]\}_{i= [n]}$ be the output of the detector for a single scan.
With a slight abuse of notation we also refer to $\boldx$ as a {\em scan}. 
Typically  a TOFMS  experiment consists of many scans which are later processed (simply averaged) to obtain a more accurate estimate of the spectrum. 
Let $\boldx^{(l)}[i]$ be the $l^{th}$ scan.
Define the {\em true spectrum}, $\bar{\boldx}[i]$, as the average of infinitely many scans, i.e., $\bar{\boldx}[i] = \lim_{N \to \infty} \frac{1}{N}\sum_{l=1}^{N}\boldx^{(l)}[i]$.
Each scan $\boldx^{(l)}[i]$ is therefore a noisy version of $\bar{\boldx}[i]$.
We define the {\em trace}, $\boldy[t], \; t = 1, \dots T$, to be the observed detector response for multiple, possibly overlapping scan.
Given an observed trace $\boldy$, the goal is to find 
a \emph{good} estimate $\widehat{\boldx}$  of $\bar{\boldx}$. 
 
In what follows we treat the discrete signals as column vectors.
For $u$ and $v$ two vector of the same dimension, let $v^*$ denote the transpose of $v$ and $\langle u,v \rangle$ the scaler product of $u$ and $v$.

As a matter of convention we refer to each element of the vectors that represent the spectrum ($\boldx^{(l)}$ and $\bar{\boldx}$) as a bin and to that of the trace as a sample, e.g., $\boldy[1]$ represents the first sample of the trace.
When an ion impacts the detector it generates a bell-shaped pulse in the output of the detector. 
We refer to an observed pulse in the trace as an \emph{impact event}, or event for short.
Usually the sampling rate of the detector response is such that an event spans multiple samples. 

For pedagogical reasons, we first describe the algorithm \emph{as if}
each event could occupy only one sample and there was no time jitter, i.e., all the ions of the same species are associated with the same bin. 
We then describe the algorithm without this assumptions with some minor modifications.
All the results presented in this paper are obtained using real data from a conventional TOFMS instrument which is used to simulate the output of an ATOFMS.
The algorithm used to obtain these results is the generalized version of the algorithm.
  
\section{Measurement Scheme and the Data Model}\label{sec:data_model}

A TOF measurement from a single scan is commonly very sparse
(after removing the additive electrical noise through
 preprocessing, c.f. Section \ref{sec:simulation}).
Furthermore, a single measurement of the whole spectrum is not expensive and it can be viewed as being performed in parallel as all ions are flying in the drift region at the same time.
However, the observed signal from a single scan is too noisy and many repetitions of the same measurement are necessary to obtain an accurate estimate of the spectrum.
In a conventional TOFMS setting, the observed trace can be expressed as
$\boldy[t] = \sum_{l=1}^{N} \boldx^{(l)}[t-ln] $
where $\boldx^{(l)}$ is the detector response to the $l^{th}$ scan and $\boldx[i]$ is understood to be zero for $i\le 0$ or $i > n$.

We incorporate a simple, yet powerful, modification to this scheme \cite{moore2012statistical} (c.f. Fig. \ref{fig:ctof_vs_atof_concept}). 
Conventional TOF (\ctof{}) requires collection of many scans, each scan collected
{\em independently} with no overlap.
ATOF idea is to increase the repetition rate and allow the subsequent scans to overlap.
\begin{figure}
  \includegraphics[width=.48\textwidth]{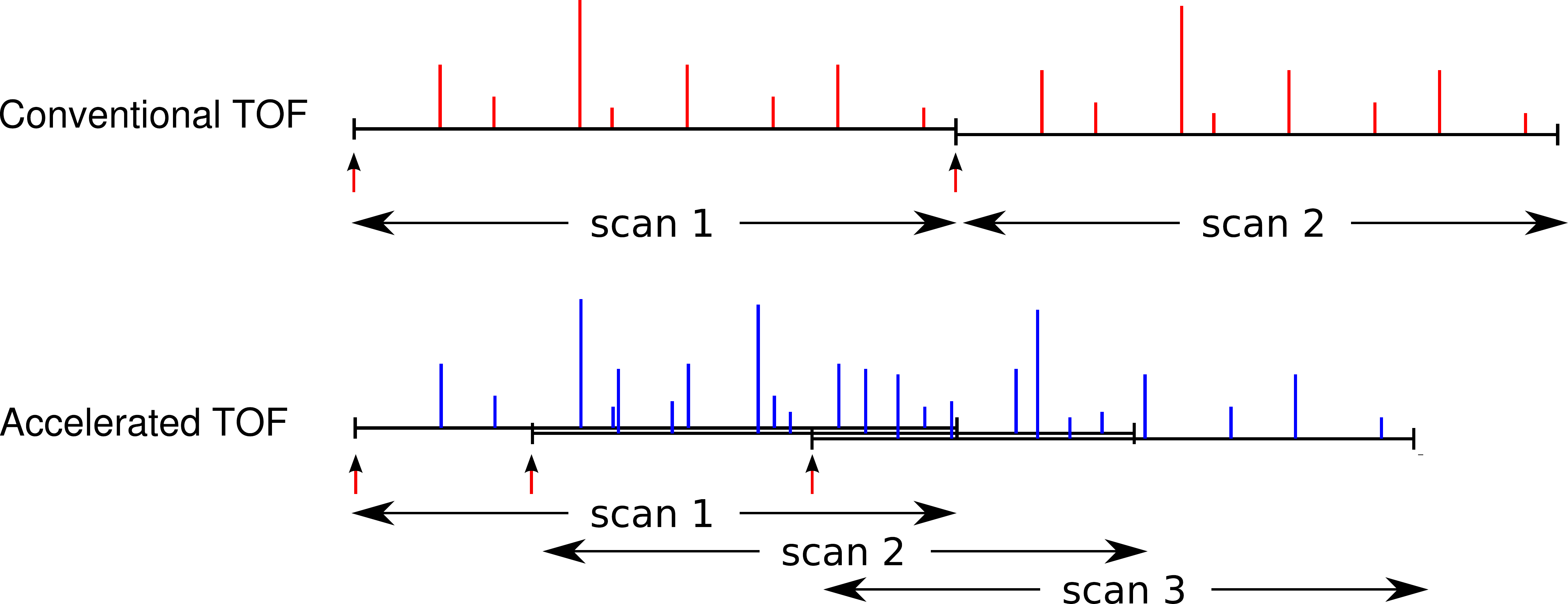}
  \caption{Difference between TOFMS and ATOFMS. In ATOFMS different scans can
  overlap resulting in shorter acquisition time for the same number of scans but
  also a {\em convoluted} observed trace.} 
\label{fig:ctof_vs_atof_concept}
\end{figure}

Define $\tau_l$, the \emph{firing time}, to be the starting time of the $l^{th}$ scan, i.e., the time when the $l^{th}$ ion packet is accelerated into the drift region. 
Define $\Delta \tau_l \equiv \tau_{l+1} - \tau_l$.
In \ctof{} $\Delta \tau_l = \Delta \tau \ge n$ to avoid overlapping between consecutive scans. 
We relax this condition and let $\Delta \tau_l$ be a random variable with  $\expect[\Delta \tau_l] = \alpha n$, for some $\alpha < 1$. 
As is the case with the HT-TOF we assume that the detector response to overlapping scans is the superposition of the individual responses,
\begin{equation}\label{eq:trace_definition} 
\boldy[t] = \sum_{l=1}^{N} \boldx^{(l)}[t-\tau_l]. 
\end{equation} 
In this case at each time $t$, $\boldy[t]$ is the superposition of multiple overlapping scans.
Assume there are a total of $N$ scans and let $0 = \tau_1 < \tau_2 < \dots < \tau_N$ be the firing times.
For a given $\boldtau = (\tau_1, \tau_2, \dots, \tau_N)$, define the matrix $A \in \reals^{T \times n}$ as
\begin{equation}\label{eq:adjacency_matrix}
A(t,i)= \left\{
\begin{array}{l l}
1 \quad &\text{if} \; \exists \; l\in [N] \; \mbox{s.t.} \; i = t-\tau_l \\
0 & \text{Otherwise.}
\end{array}
\right.
\end{equation}
The matrix $A$ can be considered as the adjacency matrix of a bipartite graph (c.f., Fig \ref{fig:bipartite_graph}), with rows of $A$ corresponding to the samples in the trace $\boldy$ and columns of $A$ corresponding to the bins on the spectrum $\boldx$.
Sample $t$ on the spectrum is connected to bin $i$ on the spectrum, i.e., $A_{ti} = 1$, if and only if for some scan $l \in [T]$ the ions from bin $i$ of the scan $\boldx^{(l)}$ arrive at time $t$ in the trace $\boldy$.
In what follows, we will refer to the neighbors of sample $t$ as $\partial t = \left\{i\in [n] \;|\; A(t,i) > 0\right\} $, and similarly to the neighbors of bin $i$ as $\partial i$.

$\boldy[t]$ can be considered as a noisy version of linear measurements of $\bar{\boldx}$, $\langle A_t, \bar{\boldx} \rangle$, with $A_t$ the $t^{th}$ row of $A$ as a column vector.
In this notation, the \ctof{} is a special case where each row of $A$ has only one nonzero element, i.e., measurement $\boldy[t]$ corresponds to a noisy observation of $\bar{\boldx}[i]$ for some bin $i$.
The structure of matrix $A$ reveals the difference between \newtof{} and \ctof{}. 
\begin{equation*}
A_{\scriptscriptstyle{\ctof{}}}= 
{
\begin{pmatrix}
\boldsymbol{1}& 0& 0& 0 \\
0& \boldsymbol{1}& 0& 0 \\
0& 0& \boldsymbol{1}& 0 \\
0& 0& 0& \boldsymbol{1} \\
\hline
\boldsymbol{1}& 0& 0& 0 \\
0& \boldsymbol{1}& 0& 0 \\
0& 0& \boldsymbol{1}& 0 \\
0& 0& 0& \boldsymbol{1} \\
\hline
\boldsymbol{1}& 0& 0& 0 \\
0& \boldsymbol{1}& 0& 0 \\
0& 0& \boldsymbol{1}& 0 \\
0& 0& 0& \boldsymbol{1} 
\end{pmatrix}
}
, \; A_{\scriptscriptstyle{ATOF}}= 
\begin{pmatrix}
\boldsymbol{1}& 0& 0& 0 \\
0& \boldsymbol{1}& 0& 0 \\
0& 0& \boldsymbol{1}& 0 \\
\boldsymbol{1}& 0& 0& \boldsymbol{1} \\
0& \boldsymbol{1}& 0& 0 \\
\boldsymbol{1}& 0& \boldsymbol{1}& 0 \\
0& \boldsymbol{1}& 0& \boldsymbol{1} \\
\boldsymbol{1}& 0& \boldsymbol{1}& \boldsymbol{1} \\
0& \boldsymbol{1}& 0& 0 \\
0& 0& \boldsymbol{1}& 0 \\
0& 0& 0& \boldsymbol{1} 
\end{pmatrix}
\end{equation*}
\begin{figure}
\includegraphics[width=.48\textwidth]{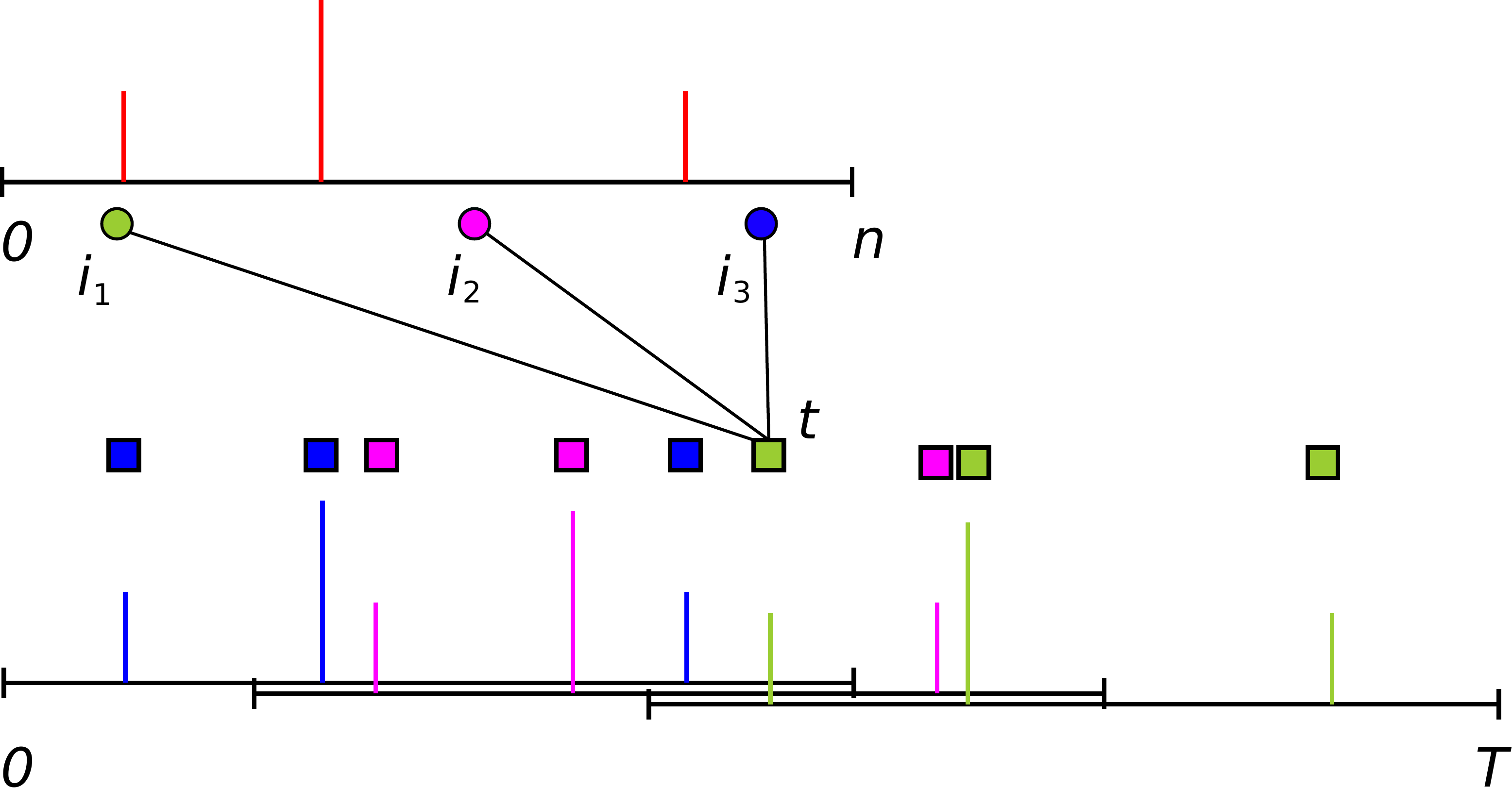}
\caption{Adjacency matrix $A$ and the corresponding bipartite graph. The signal on the top represents the spectrum and the bottom signal represents the trace. The trace is the overlapped concatenation of noisy copies of the spectrum. Nodes are color coded where blue correspond to the first scan, purple to the second, and green to the third. Neighbors of sample $t$ on the trace are those bins on the spectrum who could potentially contribute to an event at sample $t$ (again color coded). }
\label{fig:bipartite_graph}
\end{figure}

Given the trace $\boldy$ and adjacency matrix $A$ one can attempt to solve for $\bar{\boldx}$ using an ordinary least squares, $\hat{\boldx}_{{\scriptscriptstyle LS}} = {\argmin} \|A \boldx - y\|_2$ or $\ell_1$-regularized least squares $\hat{\boldx}_{{\scriptscriptstyle LASSO}} = {\argmin} \|A \boldx - y\|_2 + \lambda \|\boldx\|_1$ \cite{tibshirani1996regression}. 
However, simulation results demonstrate poor performance for both these methods. 
The reason  lies in the choice of the objective function.
Sum of square residuals approximates the negative log likelihood when
the measurement noise is additive Gaussian.
However, TOFMS is dominated by shot-noise which is signal dependent and non-additive.
Similar issues arises in applications like photon-limited imaging where the observations are again shot-noise limited. 
Regularized maximum likelihood approaches proved effective in these settings \cite{harmany2010spiral}.
Here we propose a stochastic model for the observation $\boldy$ and present an algorithm that optimizes the $\ell_1$-regularized log likelihood.

\begin{figure}
\center
\includegraphics[width=.5\textwidth]{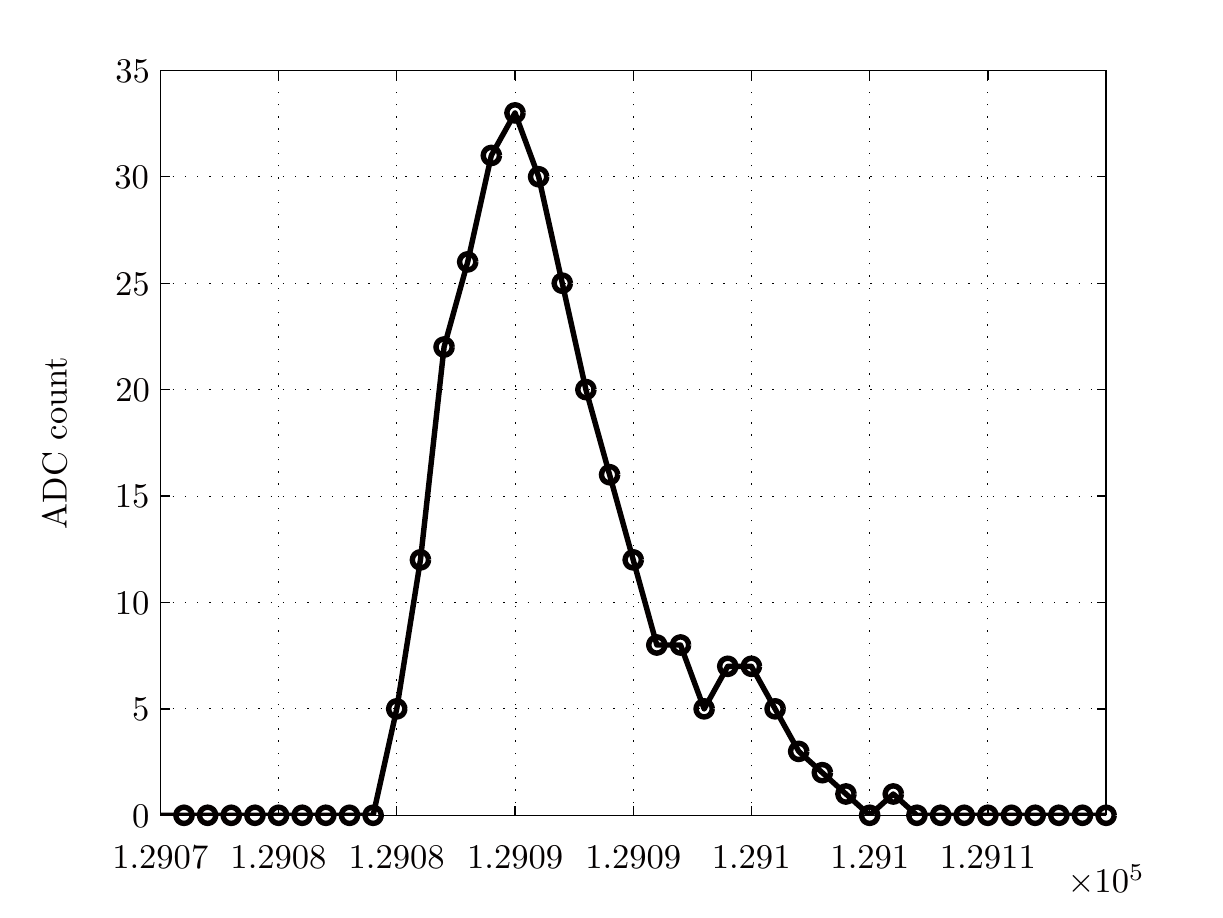}
\caption{A sample event from the trace}
\label{fig:sample_event}
\end{figure}

Figure~\ref{fig:sample_event} shows a sample observed event plotted as the output of ADC vs bin number. 
As in this figure, each such event can span multiple samples. 
However, to simplify the presentation of the algorithm we first assume that each ion impact can occupy only one sample on the trace.
The algorithm is extended to the realistic case of events spanning multiple samples in the next section.
All the experimental results presented in this paper also corresponds to this general case.
However, we have not extended the theoretical results of our paper to the general case.

Define $\boldw \in \reals^n$ such that $\boldw[i]$ is the average number of ions that impact the detector at bin $i$ for a single scan.
In each scan a large number of molecules of each species enter the instrument. 
However, each molecule has a slight chance of passing all the stages of the instrument and reaching the detector.
Therefore, it is a natural choice to assume that the number of ions that impact the detector at time $i$ follows a Poisson distribution with mean $\boldw[i]$.
Figure~\ref{fig:epmf_num_impacts} shows the estimated empirical probability mass function (EPMF) for the number of ions that impact the detector.
Note that here we mentioned \emph{estimated} EPMF since we do not directly observe the number of ion impacts.
What we observe instead is the current at the output of the detector which has an arbitrary scaling.
We estimate the number of ion impacts as follows.
We start with ten thousands acquisition of the same sample and identify a set of \emph{rare} ions as ions that are observed in more than $0.1\%$ but less than $1\%$ of acquisitions.
These ions correspond to chemical species with low concentrations and have small probability of having multiple ion impacts in any acquisition.
We take the median area under the pulse (weight) for these ions as the estimate for the weight of a single ion impact.
We then normalize the weight of all events by this estimate and round it to the closest integer.
Figure~\ref{fig:epmf_num_impacts} shows the estimated empirical probability mass function (EPMF) for the number of ions that impact the detector and its maximum likelihood Poisson fit.

This result indicates that a Poisson model for $\boldy$ is inadequate.
In particular, a Poisson random variable, having its mean and variance tied together, cannot explain the observed variation in the tail.
We assume that the detector response is additive for multiple concurrent impacts \cite{brock1998hadamard}.
Furthermore, as suggested in \cite{wiza1979microchannel} we assume each ion impact generates a cumulative ADC count which itself is an exponentially distributed random variable. 
Therefore, conditioning on the event that $k$ ions impact the detector at a certain time the cumulative ADC count has an Erlang distribution with the shape parameter $k$.
Figure~\ref{fig:epdf_event_weight} shows the empirical probability density function of normalized event weight and its maximum likelihood fit of the Poisson+Erlang model.
Comparing Figs~\ref{fig:epdf_event_weight} and \ref{fig:epmf_num_impacts} reveals that the Poisson+Erlang is a much more satisfactory model for this data.

\begin{figure}
\center
\includegraphics[width=.5\textwidth]{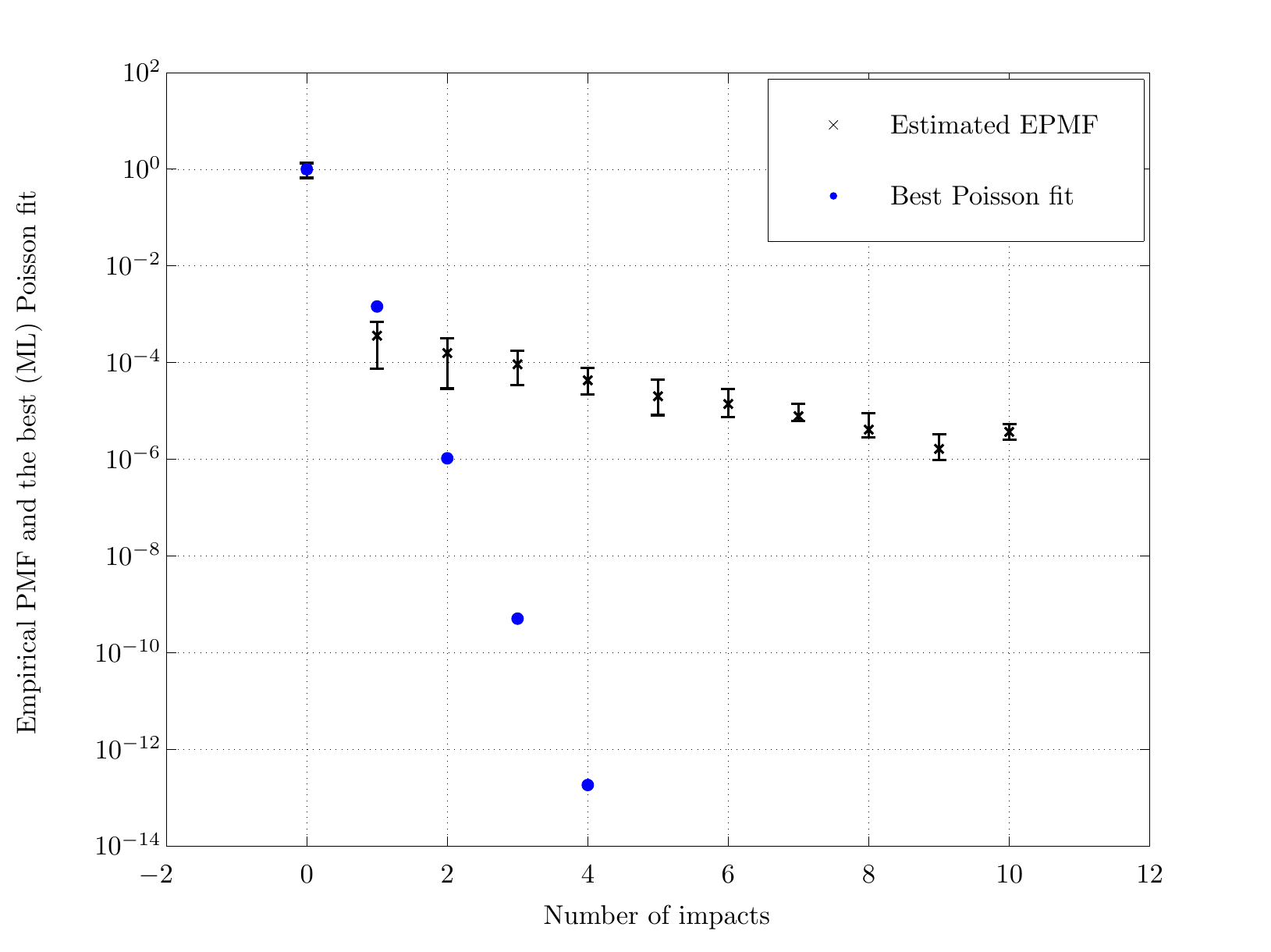}
\caption{Estimated empirical probability mass function (EPMF) for the number of ions that impact the detector and the best (ML) Poisson fit for this data. See Section \ref{sec:generalization} for details of estimating EPMF. This figure shows that a Poisson random variable is not adequate for modeling the number of impacts.}
\label{fig:epmf_num_impacts}
\end{figure}
\begin{figure}
\center
\includegraphics[width=.5\textwidth]{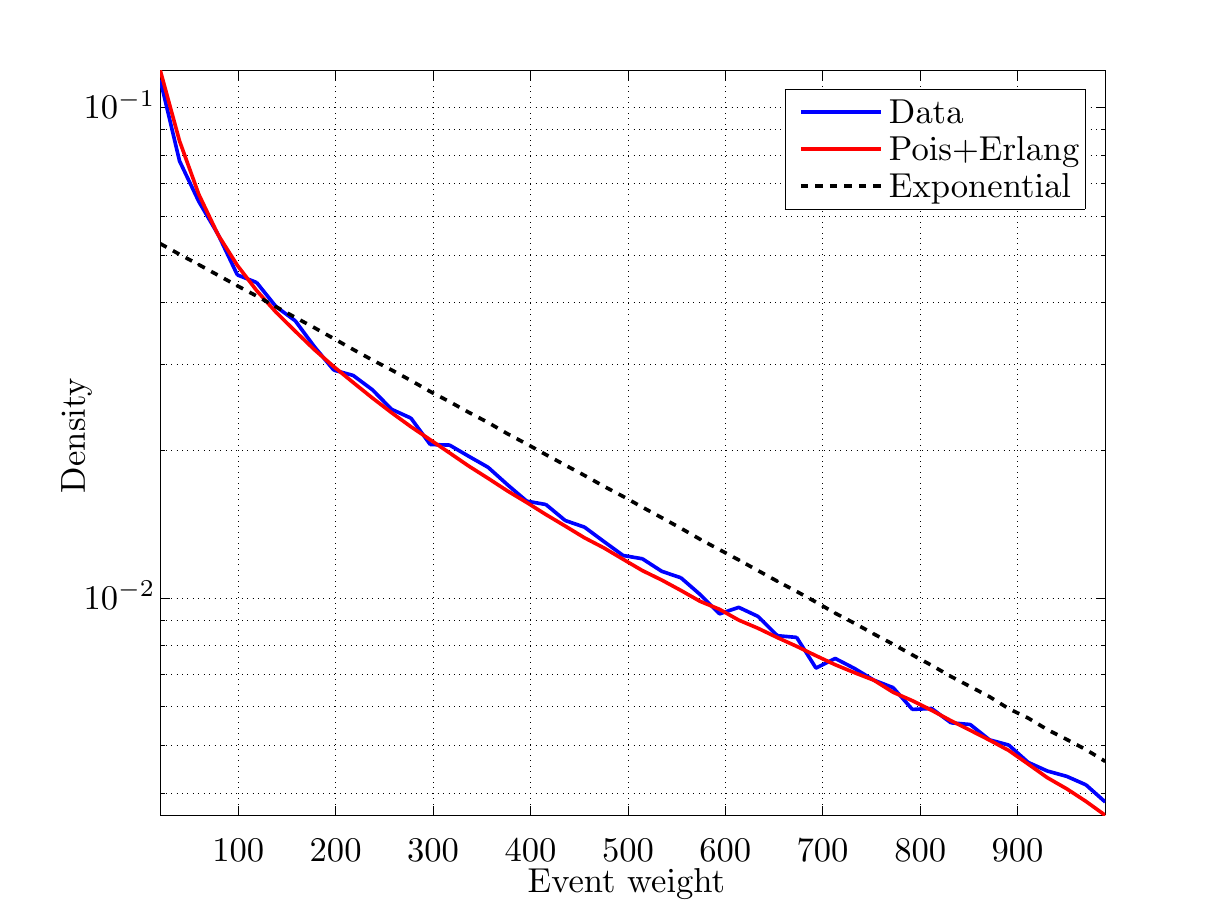}
\caption{Empirical pdf of the normalized observed event weight (area under pulse) and a Poisson+Erlang model fitted to this data.}
\label{fig:epdf_event_weight}
\end{figure}

Let $\mu$ be the mean of the exponential random variable describing the detector response. Given $\boldw$ and $A$ the probability density function of $\boldy[t]$ is
\begin{align}\label{eq:density}
\bP(\boldy[t]|\boldw, A) &= e^{-\langle A_t, \boldw\rangle - w^0} \delta(\boldy[t]) \\
& \quad + \sum_{k=1}^{\infty}{E_{k,\mu}(\boldy[t]) P_{\langle A_t, \boldw\rangle + w^0}(k) },
\end{align}
where $\delta(\cdot)$ is the Dirac delta representing a probability mass at zero and $E_{k,\mu}(\cdot)$ and $P_{\lambda}(\cdot)$ are the Erlang PDF and Poisson PMF defined as 
\begin{align} 
E_{k,\mu}(y) =  \frac{y^{k-1} \exp(-\frac{y}{\mu})}{\mu^{k}(k-1)\!\,!}, \\
P_{\lambda}(k) = \frac{\exp(-\lambda)\lambda^{k}}{k\!\,!}.
\end{align}
$w^0 > 0 $ is a constant accounting for the spurious ion impacts observed at the detector.

Assuming that each ion impact is observed in only one sample of the trace, different trace samples are the result of different ion impacts.
Hence, given $\boldw$ and $A$ the observed responses in different samples can be considered independent, namely
\begin{align*}
\bP(\boldy|\boldw, A) = \prod_{t=1}^T \bP(\boldy[t]|\boldw, A)
\end{align*}
Let $\mathbb{I}(\cdot)$ denote the indicator function, $\mathbf{1} \in \bR^T$ the vector of all ones, and define $\log(0)\mathbb{I}(\text{\small\sc FALSE}) = 0$.
Hence the negative log-likelihood function takes the form 
\begin{align}\label{eq:neg_log_likelihood}
&\ell(\boldw; \, \boldy, A) =  \langle \mathbf{1}, A \boldw \rangle + N w^0 \nonumber \\
&- \sum_{t=1}^{T} \log\left(\sum_{k=1}^{\infty}
\frac{\boldy[t]^{k}(\langle A_t,\boldw\rangle + w^0)^{k}}{(k-1)\!\,! \, k\!\,! \,\mu^{k}} \right) \mathbb{I}(\boldy[t] > 0) 
\end{align}
for $\boldw \ge 0$, and infinity otherwise.
Define $F \subseteq [T]$ as $F = \{t \in [T] |\, \boldy[t] > 0\}$.
Alternatively, for $\boldw \ge 0$ the log-likelihood function can be written as
\begin{align}\label{eq:neg_log_likelihood_2}
& \ell(\boldw; \, \boldy, A) =  N \| \boldw\|_1 \nonumber \\
 & - \sum_{t\in F}\log\left\{\sqrt{\langle A_t,\boldw\rangle + w^0}\; 
 \bessel_1 (\frac{2}{\sqrt{\mu}}\sqrt{\boldy[t]\langle A_t,\boldw\rangle + w^0}) \right\},
\end{align}
where $\bessel_1$ is the modified Bessel Function of the first kind, and we dropped the terms $\frac{1}{2} (\log(\boldy[t])- \log(\mu))$ and $N w^0$ which do not depend on $\boldw$.

We note that the negative log-likelihood function \eqref{eq:neg_log_likelihood_2} is strictly convex in
$\boldw$ which is remarkable given the existence of the \textit{hidden} variable $k$. 
\begin{lemma}\label{lem:neg_log_liklihood_convex}
  The function $\ell_{y, A}(\boldw)$ is strictly convex for $\boldw \ge 0$. 
\end{lemma}
\begin{proof}
The proof is immediate using the following theorem due to Findling
\cite{findling1995family}.
\begin{theorem} {(Findling 95)}
  The function $x \bessel(x)$ is strictly log-concave on $\{x \in \reals \;|\; x >0 \}$
\end{theorem}
\end{proof}

A simple transformation of the negative log-likelihood function
\eqref{eq:neg_log_likelihood_2} is insightful.
Let $ \lambda = 2N\mu, \quad \boldtw = \boldw/\mu$, then for $\boldtw \ge 0$,
\begin{align}\label{eq:regularized_NLL_decomposition}
\ell(\boldw; \, \boldy, A) =  \widetilde{\ell}(\boldtw; y, A) + \lambda \|\boldtw\|_1,
\end{align}
where
\begin{equation*}
  \widetilde{\ell}(\boldtw; y, A) =  - \underset{t \in F}{\sum}
  {\log\left(\sqrt{y[t] \underset{i \in \partial t}{\sum}  \boldtw_i} \;
  \bessel_1 \left(\sqrt{y[t] \underset{i \in \partial t}{\sum} \boldtw_i}\right)\right) } .
\end{equation*}
A few remarks are in order. 
First, note that the scaling of the variable $\boldtw$ is irrelevant for our purpose and only the relative values are important.
Hence, this is a single-parameter representation of the NLL function.
This is of great importance for practical systems where optimal tuning of multiple parameters in different operational scenarios can be complicated and require additional expertise.
Second, the single tuning parameter appears as the multiplicative factor in front of the $\ell_1$ regularization term.
Given our intuition about the effect of $\ell_1$ regularization \cite{tibshirani1996regression, zhao2006model} the parameter governs the sparsity of the estimate $\boldtw$, i.e., the number of species that appear in the output.
In what follows we free the parameter $\lambda$ from our original interpretation of it as the product $2 N \mu$ and refer to it as the regularization parameter.
Further, we define the regularized negative log-likelihood cost function $\mathcal{C}_{\lambda}(\boldw; \, \boldy, A)$ as
%
\begin{equation}
  \mathcal{C}_{\lambda}(\boldw; \, \boldy, A) = \tell(\boldw; y, A) + \lambda \|\boldw\|_1 
\end{equation}
%
\section{Algorithm}\label{sec:algorithm}

Given the regularized negative log-likelihood cost function $ \mathcal{C}_{\lambda}(\boldw; \, \boldy, A)$ our algorithm attempts to solve the following optimization problem.
\begin{equation}\label{eq:optimization_problem}
\underset{\boldw}{\text{minimize}} \quad \tell(\boldw; y, A) + \lambda \|\boldw\|_1 \, ,\;\;\;\;\;
\text{s.t.}\;\;\;  \boldw \ge  0, 
\end{equation}
We use the now standard method of iterative soft thresholding to solve this convex but non-differentiable optimization problem.
For a doubly differentiable function $f: \mathcal{D} \subset \bR^n \rightarrow \bR$ let $\nabla f$ and $\nabla^2 f$ be the gradient and Hessian of $f$ respectively. 
Let $\gamma > 0$ be such that $\|\nabla^2 \tell(\boldw; \, \boldy, A)\|_2 < \gamma^{-1}$ for $\boldw \ge 0$. 
It is easy to see that $\gamma$ exists because of the presence of the chemical noise term $w^0$.
We call the parameter $\gamma$ the step size as we use it to scale the steps the algorithm takes in each iteration.
Let $\boldw^{(k)}$ be our estimate of $\boldw$ at step $k$.
Then we can compute an upper bound for the cost function $\mathcal{C}_{\boldy, A, \lambda}(\boldw)$ as follows.
\begin{align}
\mathcal{C}_{\lambda}(\boldw; \, \boldy, A) &\le
\tell(\boldw^{(k)}; \, \boldy, A) + \boldw^* \, \nabla
\tell(\boldw^{(k)}; \, \boldy, A) \nonumber\\
& + \gamma^{-1} \|\boldw - \boldw^{(k)}\|_2^2 + \|\boldw\|_1. \label{eq:cost_func_upper_bound}
\end{align}
Equation ~\eqref{eq:cost_func_upper_bound} provides an approximation for $\mathcal{C}_{\boldy, A, \lambda}(\boldw)$ when $\boldw$ is in a small neighborhood of $\boldw^{(k)}$. Minimizing the right-hand-side of Eq.~\eqref{eq:cost_func_upper_bound} with respect to $\boldw$ as a surrogate for the actual cost function results in
\begin{equation}
\boldw^{(k+1)} = \eta_{\theta}\left( \boldw^{(k)} - \gamma \nabla \tell(\boldw^{(k)}; \, \boldy, A) \right).
\end{equation}
where $\eta_{\theta}(\cdot)$ is the soft thresholding function, $\eta_{\theta}(x) = (|x|-\theta)_+$ with $(\cdot)_+$ being the positive part and $\theta \propto \lambda^{-1}$. 
Note that this is the \emph{one-sided} soft thresholding function which differs from the two-sided soft thresholding function by mapping all negative values to zero.
%
%
From Eq.~\eqref{eq:neg_log_likelihood_2}, $\nabla \tell_{\lambda}(\boldw; \, \boldy, A)$ can be calculated as
\begin{align}
& \nabla \tell_{\lambda}^*(\boldw; \, \boldy, A) = - \sum_{t\in F}\bigg( \frac{1}{2\langle A_t,\boldw\rangle} \nonumber \\
& + \frac{\bessel_0 \left( 2 \sqrt{\boldy[t] \langle A_t,\boldw\rangle}\right) +\bessel_2 \left( 2 \sqrt{ \boldy[t] \langle A_t,\boldw\rangle}\right)}{2 \sqrt{\boldy[t] \langle A_t,\boldw\rangle} \bessel_1 \left( 2 \sqrt{\boldy[t] \langle A_t,\boldw\rangle}\right)} \bigg) A_t.
\end{align}
Given the gradient of the log-likelihood the algorithm is as follows.
\begin{center}
     \begin{tabular}[c]{l}
     \hline
     \textbf{Algorithm} \\
     \hline
     \textbf{Input:} trace $\boldy$, firing times $\boldtau$, and constants ($\theta_0$, $\theta_1$, $\mu$)\\
     \textbf{Output:} estimated spectrum $\widehat{\boldx}$\\
     1:\quad Calculate the adjacency matrix $A$ as in Eq.~\eqref{eq:adjacency_matrix}.\\
     2:\quad  Set $\boldw^{(0)} = 0$, $\theta = \theta_0 + \theta_1$. \\
     2:\quad Repeat until stopping criterion is met\\
       \quad \qquad $\theta \leftarrow \theta_0 + \frac{1}{k^2} \theta_1$\\
       \quad \qquad $\boldw^{(k+1)} \leftarrow \eta_{\theta}\left( \boldw^{(k)} - \gamma \nabla \tell_{\lambda}(\boldw^{(k)}; \, \boldy, A) \right).$ \\
     3:\quad Set $\widehat{\boldx} = 0$\\
     4:\quad For $t \in F$\\
     \quad \qquad $i_* = \underset{i \in \partial t}{\arg\max} \; \boldw[i]$ \\
     \quad \qquad $\widehat{\boldx}[i_*] = \widehat{\boldx}[i_*] + \boldy[t]$ \\
     5:\quad Return $\widehat{\boldx}$.\\
      \hline
      \end{tabular}
\end{center}

Step $4$ in the algorithm is worth noting. 
It was mentioned that each observed event $t$ has a set of possible bins on the spectrum it can be caused by, $\partial t$ (c.f., Fig.~\ref{fig:bipartite_graph}).
The problem of estimating the spectrum from the observed \newtof{} trace can be thought of as assigning each observed event to one of its neighbors which is the true cause of the event.
This framing of the problem enables us to terminate the slow first order optimization method as soon as we are confident about the likely assignment of an event.
In the generalized algorithm which is concerned with the case of multi-sample events this technique proves instrumental in decreasing the bias in the estimated spectrum.

\section{Experimental Evaluation}\label{sec:simulation}

In this section we present performance evaluation results for the \newtof{} algorithm.
We use a commercial TOFMS instrument for data collection and obtain $10,000$ scans for a high concentration multimode chemical sample using conventional TOFMS technique. 
Each scan is \unit{100}{\micro\second} in length sampled at \unit{25}{\pico\second} intervals.
Therefore, in our notation $n=4\times 10^5$.
The average of these ten thousand scans is considered the ground truth. 

For evaluation, we use a random subset of $1,000$ scans and simulate \newtof{} using these scans as follows. 
First, we construct a vector of firing times, $\boldtau = (\tau_0, \tau_1, \dots, \tau_{N-1})$ by setting $\tau_0=0$ and choosing $\Delta \tau_i \equiv \tau_i-\tau_{i-1}$ uniformly at random from the interval $[\Delta\tau_{\min}, \Delta\tau_{\max}]$. 
Given $\boldtau$, the trace is constructed using unaliased scans as prescribed by Eq.~\eqref{eq:trace_definition}.  

\begin{figure}
\center
\includegraphics[width=.5\textwidth]{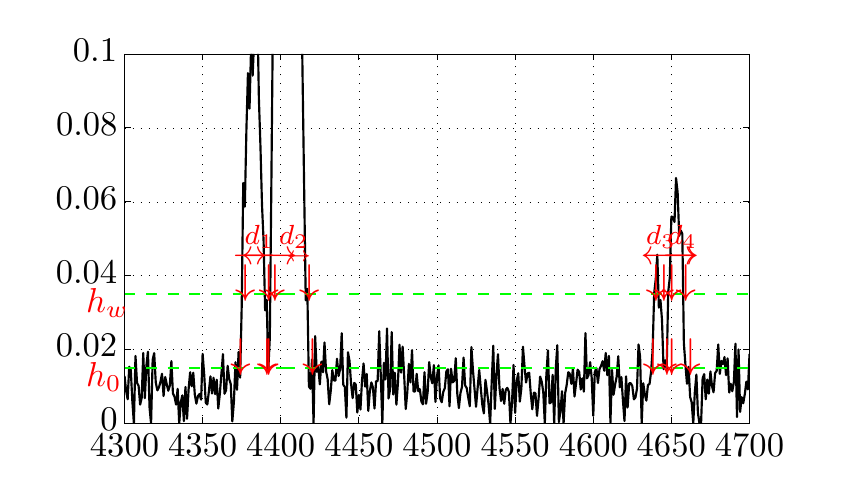}
\caption{Preprocessing the data. There are four pulses that exceed the threshold $h_{w}$ from which three pass the minimum width condition $d_i \ge d_{\min}$ ($1$, $2$, and $4$). The markers at the $h_0$ level indicate the start and end of each marked event.}
\label{fig:thresholding_example}
\end{figure}
\begin{figure}
\center
\includegraphics[width=.4\textwidth]{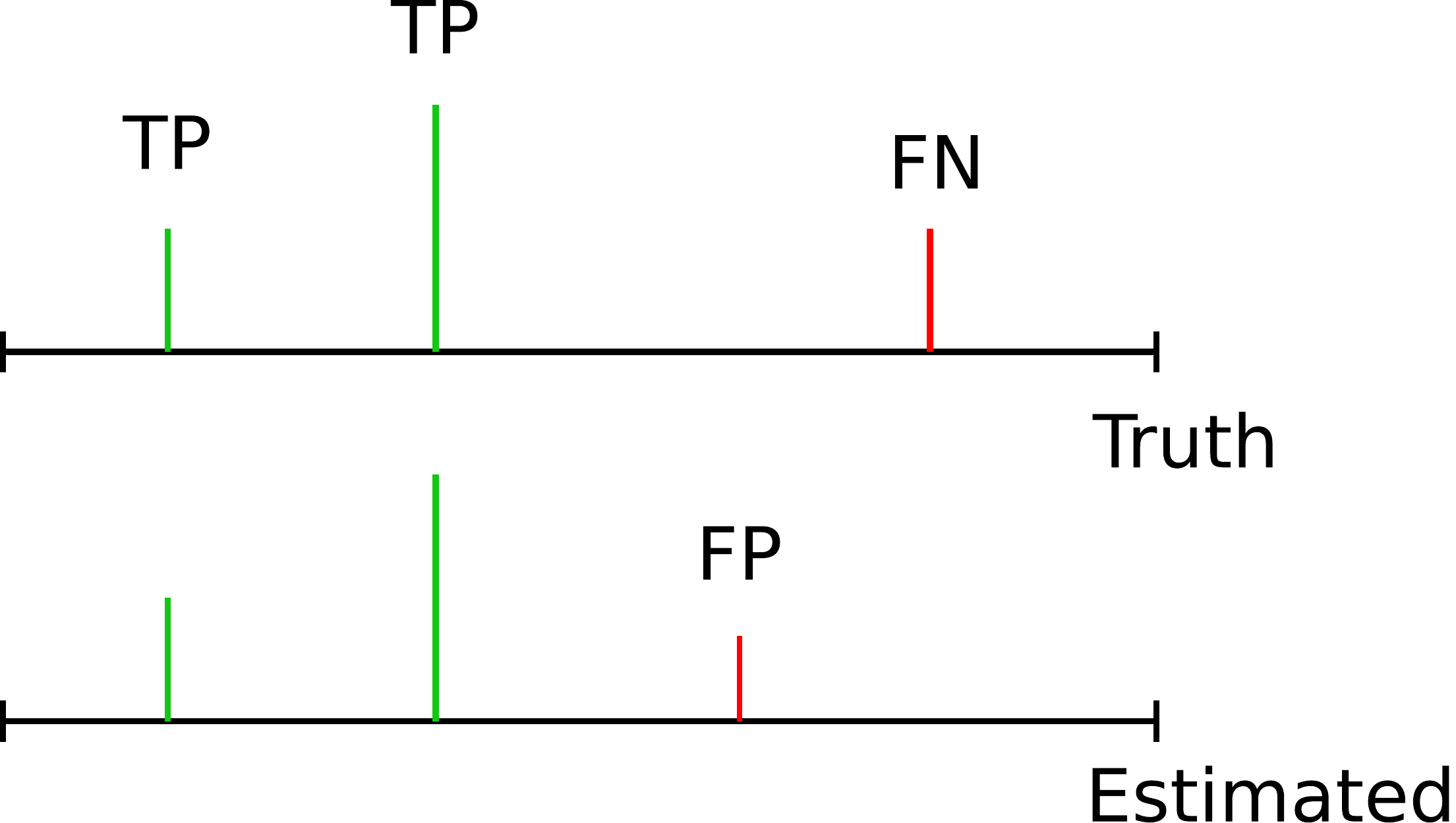}
\caption{Illustration of true positive (TP), false positive (FP), and false negative (FN). }
\label{fig:tp_fp_fn_illustration}
\end{figure}
\begin{figure}
\center
\includegraphics[width=.5\textwidth]{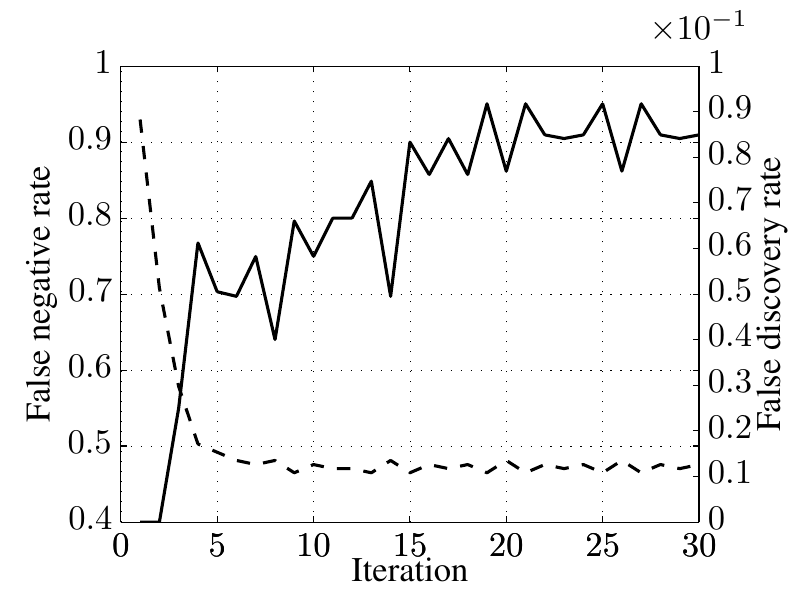}
\caption{False negative rate (dashed) and false discovery rate (solid) vs. iteration.}
\label{fig:fnr_fdr_vs_iteration}
\end{figure}
\begin{figure}[t]
\center
\includegraphics[width=.5\textwidth]{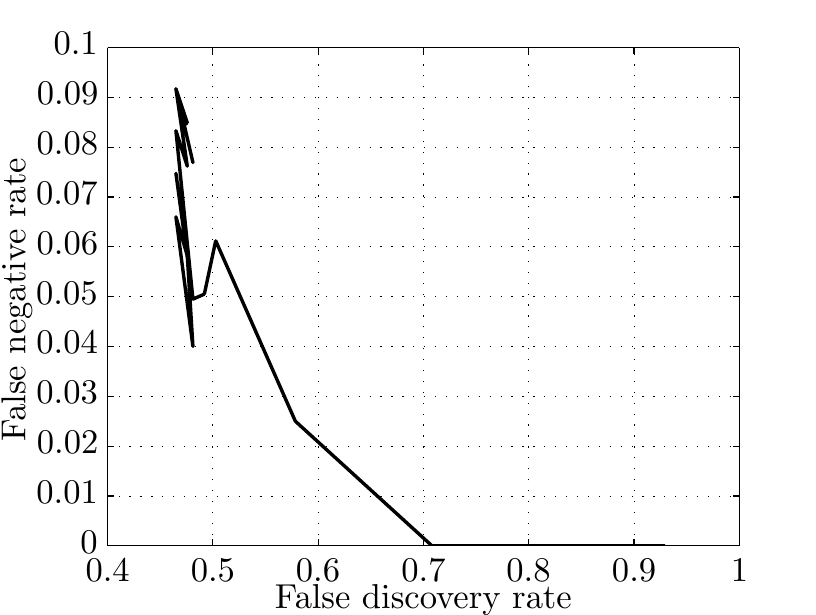}
\caption{False negative rate vs. false discovery rate.}
\label{fig:fdr_fnr_iteration}
\end{figure}

In addition to the aliasing effect, the trace is corrupted by noise. 
Henceforth, we preprocess the trace before applying the reconstruction algorithm by setting the trace to zero unless it is likely to be the result of an ion impact. 
As mentioned before, the detector response to each ion impact is a bell shaped pulse which spreads across multiple samples.
These pulses are corrupted by electrical noise that can be modeled as additive noise. 
However, the electrical noise level is significantly smaller in magnitude compared to the detector response to an ion impact and henceforth ion impacts can be marked with high confidence. 

We select the potential ion impacts through the following procedure. 
Define three constants $h_0$, $h_w$, and $d_{\min}$ and label a pulse as a potential ion impact event if the width of the pulse at hight $h_{w}$ is greater than $d_{\min}$. 
If an event satisfies this criterion the support interval  of the event is determined by thresholding the signal at level $h_0$. 
Figure \ref{fig:thresholding_example} demonstrates this procedure through an example. 
In this figure there are four pulses that exceed the threshold $h_{w}$ in peak magnitude.
From these pulses $d_1$, $d_2$, and $d_4$ satisfy the minimum width condition of $d_i \ge d_{\min}$ at height $h_w$. 
Henceforth, after the preprocessing there are three valid events with the start and end times marked at level $h_0$.
We set the trace equal to zero wherever it does not support a valid event.
After the preprocessing step, the trace can be represented as a list of events whereby each event describes a single pulse in the trace.
Note that observed traces are outputs of an ADC and have an arbitrary scaling.
We keep this scaling but note that the absolute value of the amplitude of the trace and the corresponding parameters like $h_w$ and $h_0$ are irrelevant for our purposes.

Our procedure to identify valid events also enables us to define metrics for quantitative evaluation  of different techniques.
We take the true spectrum $\bar{\boldx}$ to be the average of all $10,000$ scans.
Let $\hx$ be an estimate of $\bar{\boldx}$. 
For some constants $h_0$, $h_w$, and $d_{\min}$ define $\bcE = \{\bar{e}_1, \bar{e}_2, \dots, \bar{e}_{\bar{m}}\}$, and $\hcE = \{\he_1, \he_2, \dots, \he_{\widehat{m}}\}$ to be the set of events in $\bar{\boldx}$ and $\hx$ respectively, obtained through the procedure described above. 
For two events $\bar{e}_i$ and $\he_j$ we say $\he_j$ matches $\bar{e}_i$ if $\bar{e}_i$ overlap with at least $50\%$ of the width of $\he_j$.
For $\bar{e}_i \in \bcE$ we say $\bar{e}_i$ is a false negative if none of the events in $\hcE$ matches $\bar{e}_i$.
For $\he_j \in \hcE$ we say $\he_j$ is a true positive if there exist $\bar{e}_i \in \bcE$ such that  $\he_j$ matches $\bar{e}_i$ and we say $\he_j$ is a false positive if it does not match any event in $\bcE$.
See Fig~\ref{fig:tp_fp_fn_illustration} for an illustration of these concepts.

Let {\rm TP}, {\rm FP}, and {\rm FN} be the number of true positives, false positives and false negatives respectively. 
We consider false negative rate ({\rm FNR = FN/(TP+FN)}), true positive rate ({\rm TPR = TP/(TP+FN)}), and false discovery rate ({\rm FDR = FP/(FP+TP)}) as the metrics of interest.
Note that the notion of a true negative is ill-defined in this problem and hence we cannot use the false positive rate metric.
In particular, observed pulses are of different width and shapes and they can overlap.
Therefore, given an estimated spectrum the question \emph{how many pulses are not observed?} is not well-posed.

\begin{figure}
\center
\includegraphics[width=.5\textwidth]{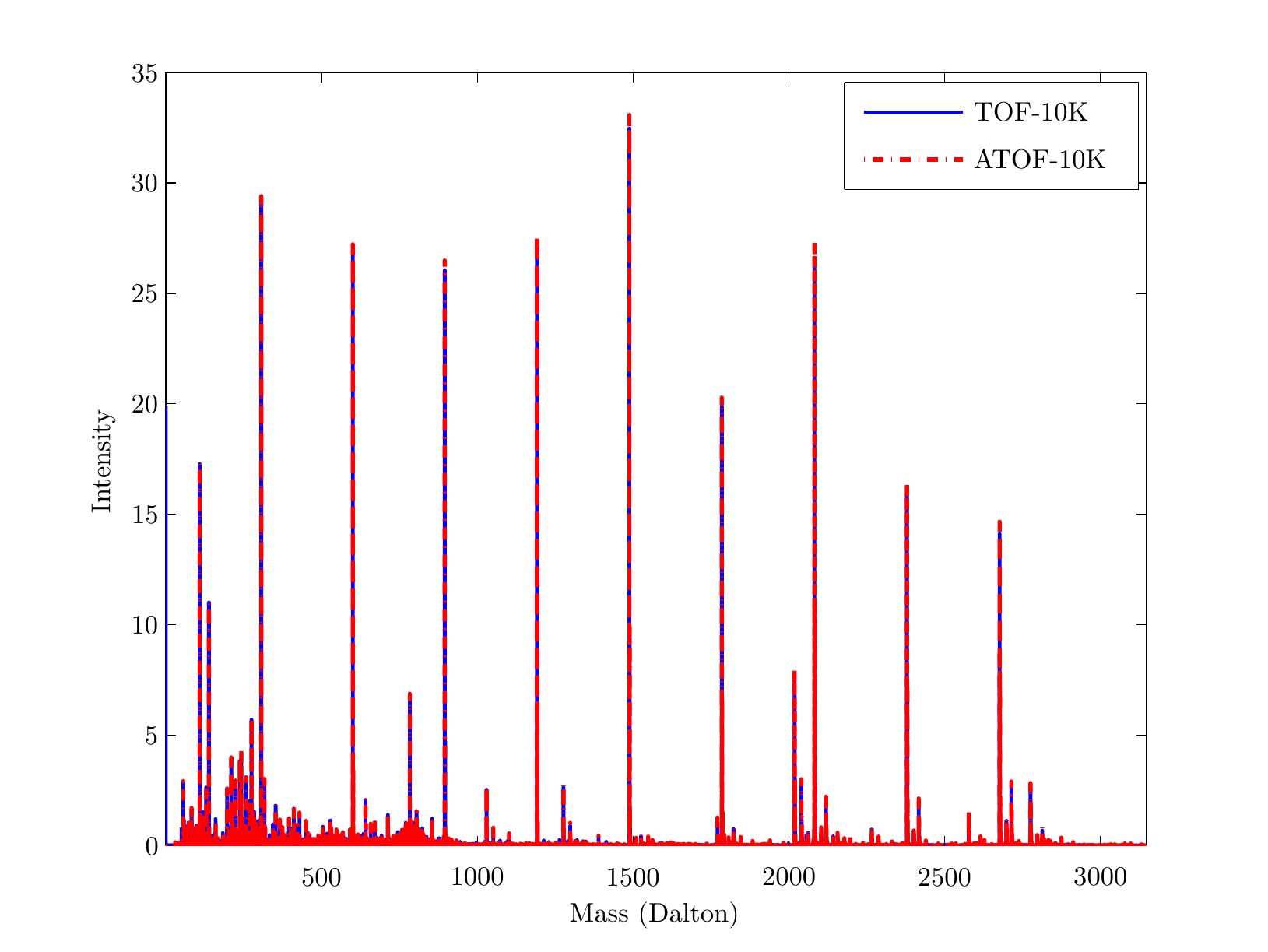}
\caption{Sample reconstructed spectrum for the acceleration factor of $10$.}
\label{fig:sample_spectrum_atof_ctof_1}
\end{figure}
\begin{figure}
\center
\includegraphics[width=.5\textwidth]{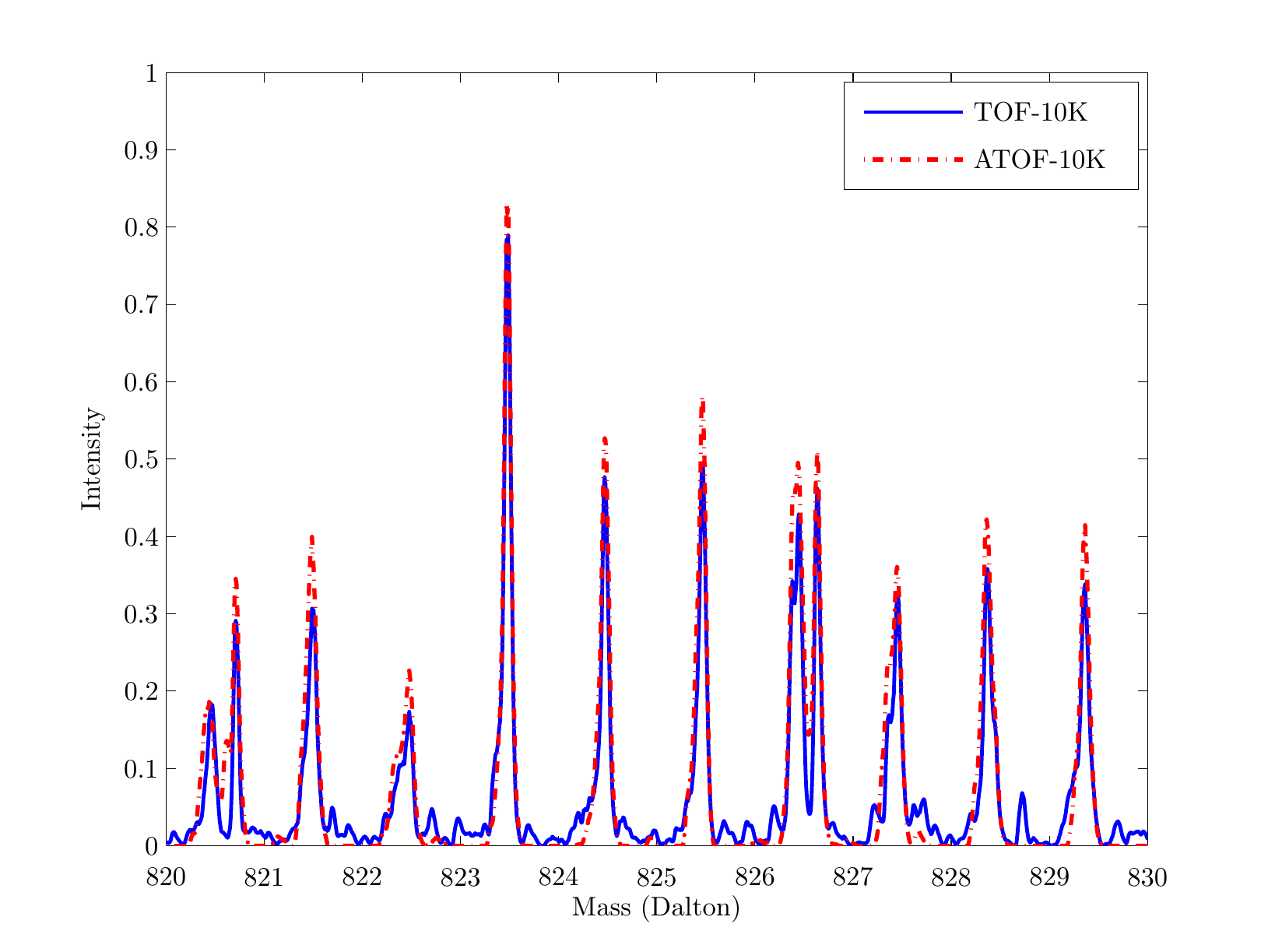}
\caption{Sample reconstructed spectrum for the acceleration factor of $10$. (different scale of Fig~\ref{fig:sample_spectrum_atof_ctof_1})}
\label{fig:sample_spectrum_atof_ctof_2}
\end{figure}

Unless otherwise stated, the parameters used to obtain the results of this section are as follows. $\mu = 225$, $\gamma = 2.5\times 10^{-3}$, $\theta_1 = 2 \times 10^{-2}$, $\theta_0 = 5 \times 10^{-4}$, $h_w^{(\text{trace})} = 2$, $h_w^{(\text{spectrum})} = 0.2$, $n = 4 \times 10^5$, $N = 10^3$, $ \Delta \tau_{\min} = 0$, $ \Delta \tau_{\max} = 2 \times 10^5$.

We also define the acceleration factor as the ratio $\equiv \frac{n}{\expect[\Delta \tau]}$.
For example, $\mathbb{E}\Delta\tau = \frac{1}{4} n$ results in an acceleration factor of $4$, which means that  \newtof{} is four times faster compared to conventional TOFMS in terms of the time it takes to collect the same number of scans. 
At the same time, each event has on average $4$ different positions on the spectrum it can be assigned to and the algorithm should be able to infer the correct position with satisfactory accuracy.  

Figure \ref{fig:fnr_fdr_vs_iteration} shows the false negative rate and false discovery rate as a function of the iteration for the \newtof{} with acceleration factor ten.
The algorithm starts with the all-zero spectrum. Hence, the false negative rate is equal to $1$ at iteration $0$ and as the algorithm proceeds the false negative rate decreases, converging to a final value of $0.47$. 
The false discovery rate on the other hand increases as the algorithm proceeds settling at a final value of about $0.085$. 
Inspecting Fig. \ref{fig:fnr_fdr_vs_iteration} suggest that the algorithm converges, in the sense of establishing the existence of ions, in about $15$ iterations.

Note that the large value of \fdr is by design.
Firstly, by setting the $h_w$ threshold (c.f. Fig~\ref{fig:thresholding_example}) very low we are requiring the algorithm to discover ions with diminutive abundance in the solutions that are observable in the ten thousand scans ground truth but very rarely appear in a random one thousand sample. 
Secondly, in a typical application of TOFMS declaring the presence of an ion that does not exist is considered a more costly mistake than missing an ion that is present.
Hence, in line with these type of applications of TOFMS instruments we choose to operate in a high \fnr and low \fdr regime.
Figure~\ref{fig:fdr_fnr_iteration} shows the same plot on the \fnr vs. \fdr plane.
This figure shows how the algorithm converges as the number of iteration increases.
Figures~\ref{fig:sample_spectrum_atof_ctof_1} and \ref{fig:sample_spectrum_atof_ctof_2} show the ground truth (solid) and reconstructed (dashed) spectrums at two different scale.
Visual inspection of the graphs indicates substantial match between the reconstructed spectrum and the ground truth.

\begin{figure*}
  \begin{subfigure}{.48\textwidth}
    \center
    \includegraphics[width=\textwidth]{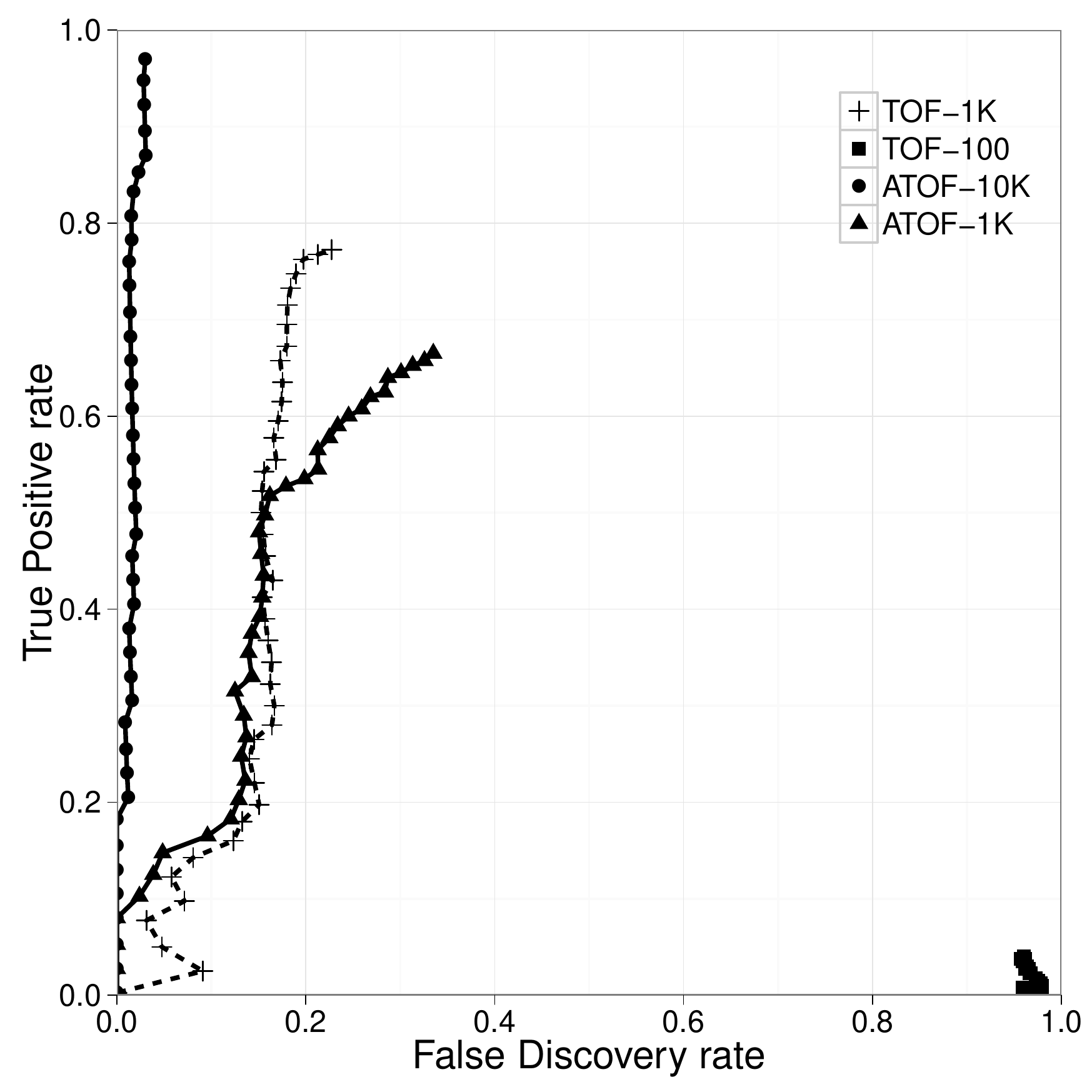} \\
    \caption{$\Delta m = 0.1$ Dalton, top $400$ peaks.}
    \label{fig:tpr_fdr_top_400_deltam_0100}
  \end{subfigure}
  \begin{subfigure}{.48\textwidth}
    \center
    \includegraphics[width=\textwidth]{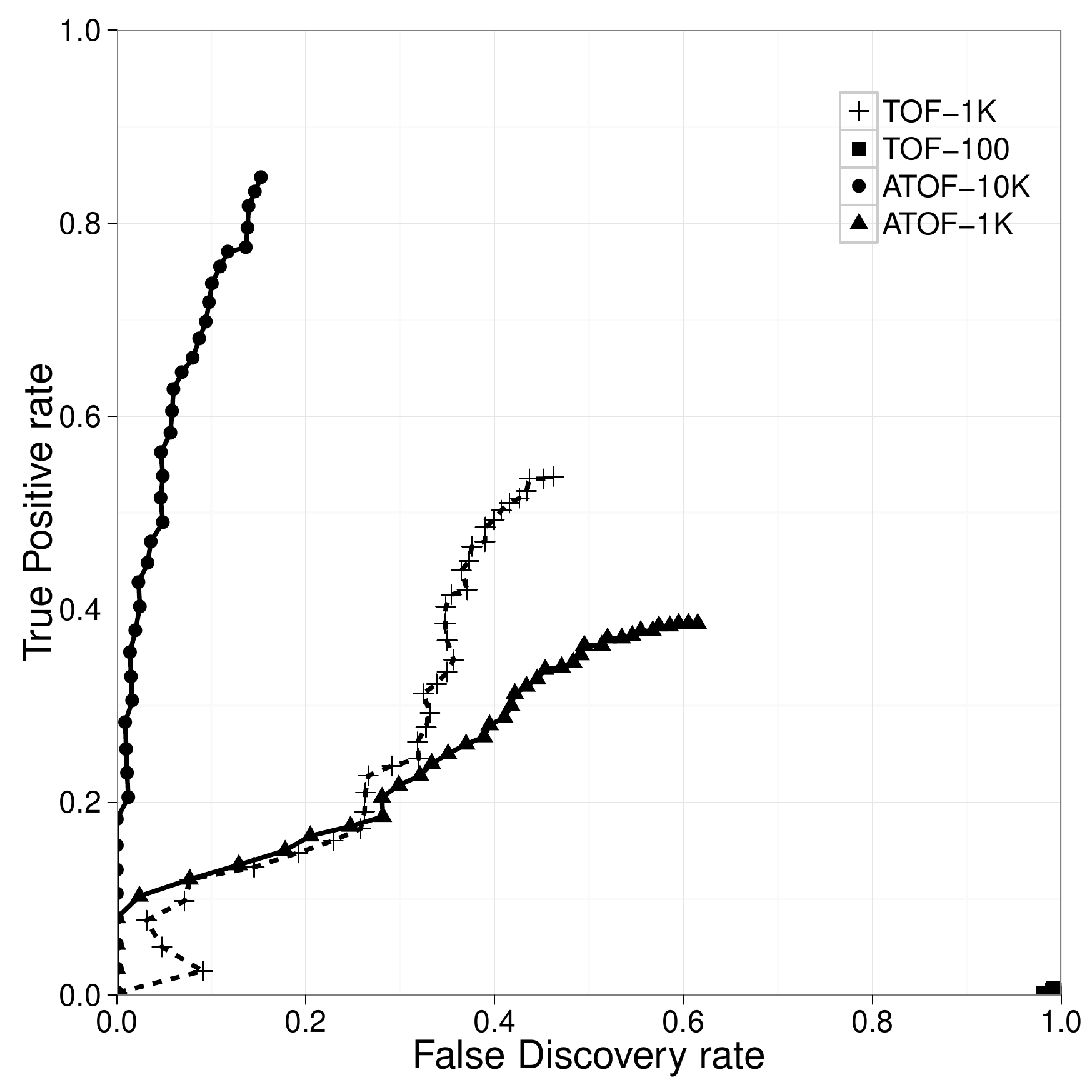} \\
    \caption{$\Delta m = 0.01$ Dalton, top $400$ peaks.}
    \label{fig:tpr_fdr_top_400_deltam_0010}
  \end{subfigure}

  \begin{subfigure}{.48\textwidth}
    \center
    \includegraphics[width=\textwidth]{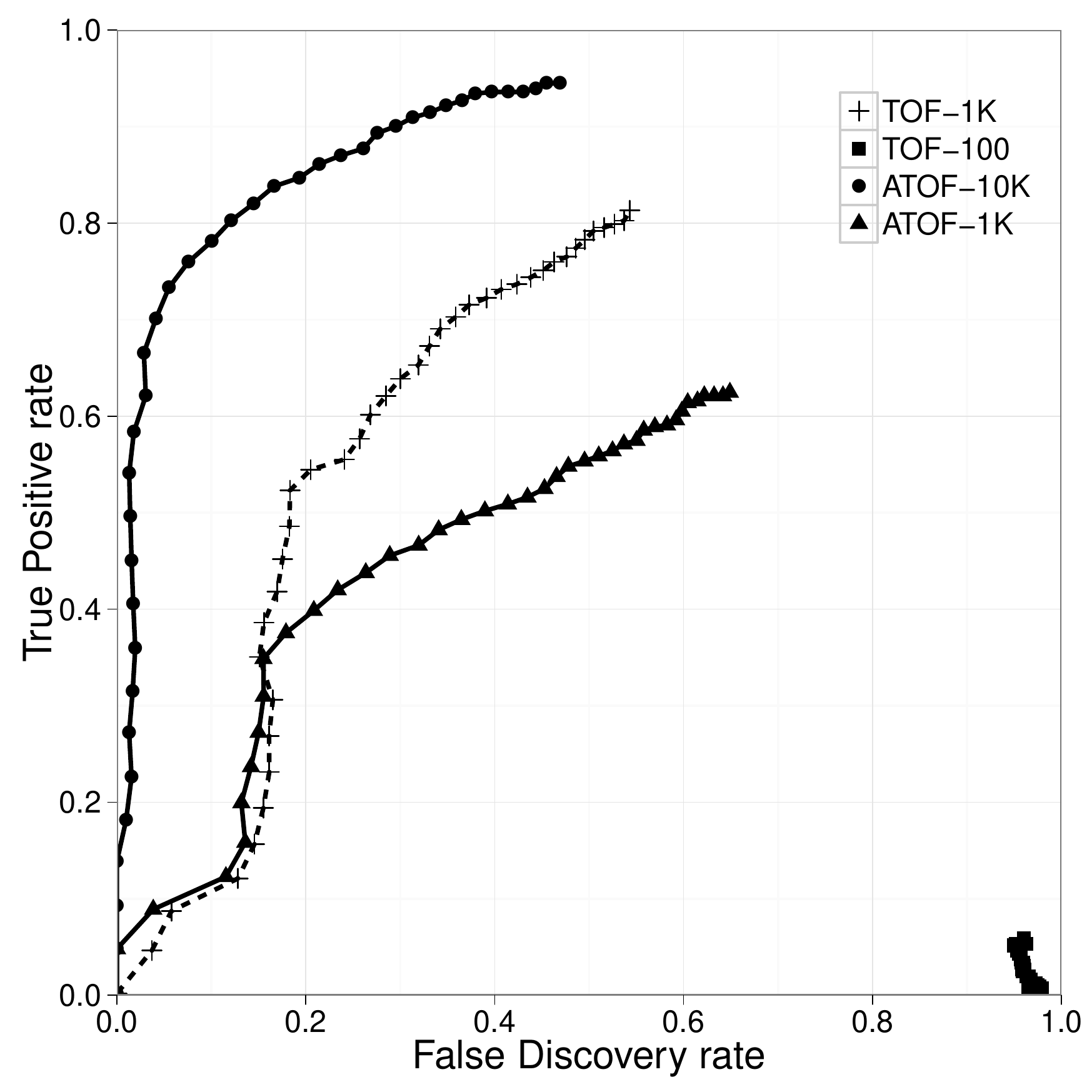} \\
    \caption{$\Delta m = 0.1$ Dalton, top $1000$ peaks.}
    \label{fig:tpr_fdr_top_1000_deltam_0100}
  \end{subfigure}
  \  \begin{subfigure}{.48\textwidth}
    \center
    \includegraphics[width=\textwidth]{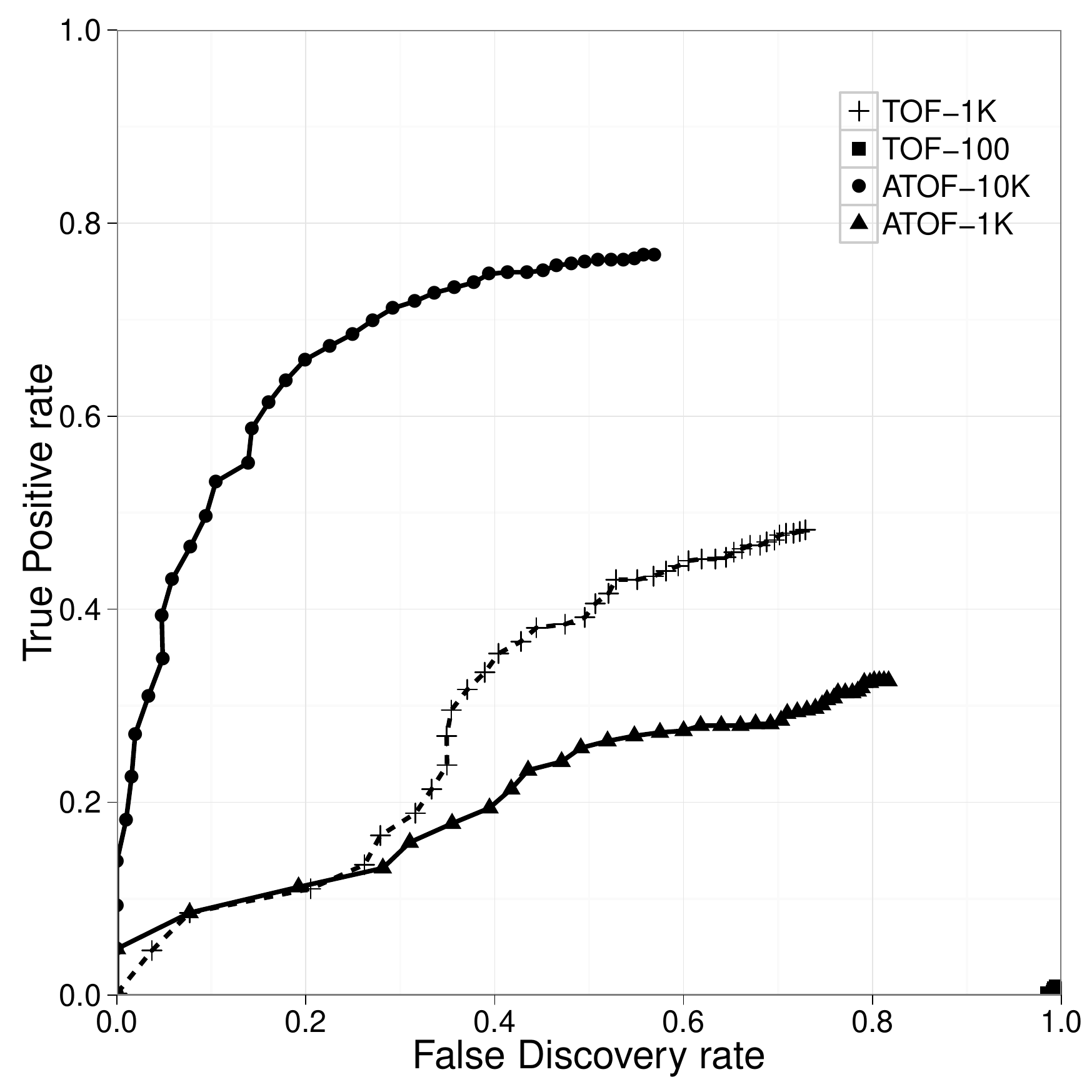}
    \caption{$\Delta m = 0.01$ Dalton, top $1000$ peaks.}
    \label{fig:tpr_fdr_top_1000_deltam_0010}
  \end{subfigure}
  \vspace{2pt}
  \caption{\tpr vs. \fdr for different values of $\Delta m$ and number of peaks. The effect of $h_w$ can be seen by comparing top and bottom plots. Lowering $h_w$ admits more smaller peaks as valid peaks in the ground truth making it more difficult to detect them. Comparing left and right plots shows the uncertainity in determinig MCR.}
  \label{fig:tpr_fdr}
\end{figure*}


The output of a TOFMS is usually used to generate a list of \emph{peaks}, i.e., the list of MCRs that are deemed present by the instrument.
Peak picking is an important task for a TOFMS instrument that can have significant effect on the instrument performance and the practice involves both publicly available methods as well as patents and trade secretes. 
These estimated peaks can be a few order of magnitudes more precise than the width of the pulses at the estimated spectrum.

Another more practically important but less transparent method for defining TP, FP, and FN is to use the list of peaks generated from the estimated spectrum.
We use the peak picking software that ships with Agilent Technologies TOFMS.
Since the estimated peaks are real-valued variables we also need to consider a precision.
With a slight abuse of notation we define two peaks to match if they are within $\Delta m$ distance of each other on the MCR scale, i.e., $\sqrt{m_1/z_1} - \sqrt{m_2/z_2} \le \Delta m$.

Figure \ref{fig:tpr_fdr} shows the TPR vs FDR for different values of $\Delta m$ and number of peaks in the ground truth.
Each plot contains four curves corresponding to \ctof{} with one hundred scans (\ctof-100), \newtof{} with one thousand scans and acceleration factor 10 (\newtof-1K), \ctof{} with one thousand scans (\ctof-1000), and \newtof{} with ten thousand scans and acceleration factor 10 (\newtof-10K).
Note that (\ctof-100) and (\newtof-1K) have the same acquisition time.
Similarly, (\ctof-1K) and (\newtof-10K) have the same acquisition time.
For these plots we run the peack picking algorithm on the ground truth spectrum and obtain a list of picks.
These picks are sorted based on their amplitude and we choose the $k$ most significant picks where $k \in \{400, 1000\}$.
We then run the peak picking algorithm on the reconstructed spectrums and obtain the list of peaks.
Similar to the process for the ground truth we keep the $k$ most significant peaks for each of these spectrums.
A peak in an estimated spectrum is considered to match a peak in the ground truth if and only if their estimated MCRs are within $\Delta m$ of each other.
For the top $400$ peaks and $\Delta m = 0.1$ (Fig.~\ref{fig:tpr_fdr_top_400_deltam_0100}, \newtof-$10$K achieves a nearly perfect reconstruction.
TOF with $1$K scans (TOF-$1$K) demonstrates acceptable but significantly inferior performance compared to \newtof-$10$K
This is while \newtof-10K and TOF-1K have the same acquisition time.
\newtof-1K and TOF-1k have comparable performance for low TPR but TOF-1k outperforms \newtof-1K for high TPR.
Note that \newtof-1K is ten times faster than \ctof{}-1K.
\ctof-$100$ scans do not have enough information to achieve any significant accuracy in this regime.
Figure~\ref{fig:tpr_fdr_top_400_deltam_0010} shows the same curves when we decrease $\Delta m$ to $0.01$, i.e., when we adopt a more restrictive definition of two peaks matching.
The overall trend is similar to that of Fig.~\ref{fig:tpr_fdr_top_400_deltam_0010}.
Figs~\ref{fig:tpr_fdr_top_1000_deltam_0100} and \ref{fig:tpr_fdr_top_1000_deltam_0010} are similar but for the top one thousand significant peaks.
The TPR and FDR degrade for all the cases since now we are expecting many more smaller peaks to be detected.
However, the relative performance of different methods and configurations remain unchanged.

\begin{figure}
\includegraphics[width=.5\textwidth]{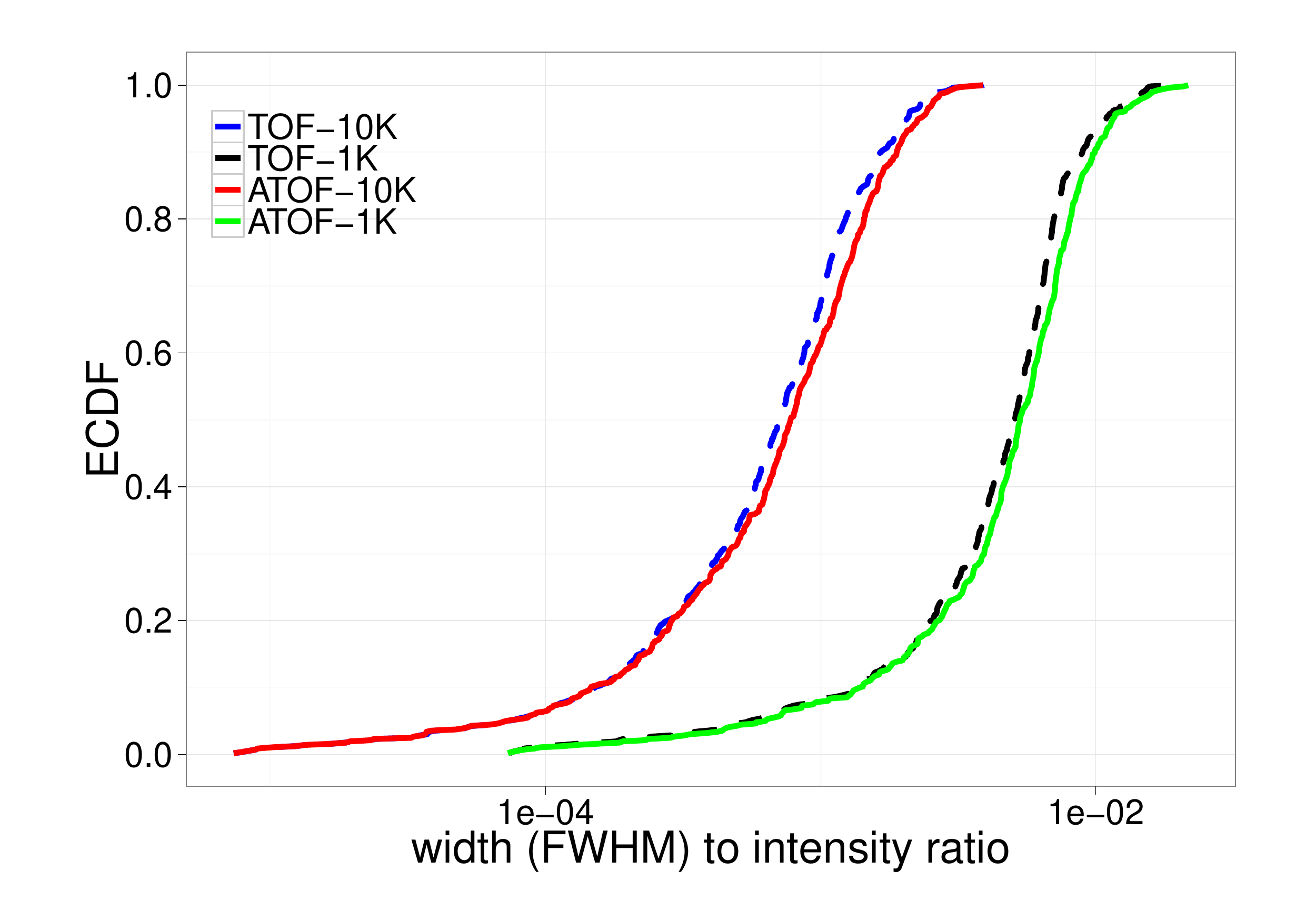}
\caption{Empirical CDF of the width (full width at half maximum) to intensity ratio for the output of \ctof{} (dashed) and ATOF (solid). ATOF results in little pulse broadening compared with \ctof{} with the same number of scans but $10$ times longer acquisition time.}
\label{fig:width_intesity_ratio}
\end{figure}
\begin{figure}
\center
\includegraphics[width=.5\textwidth]{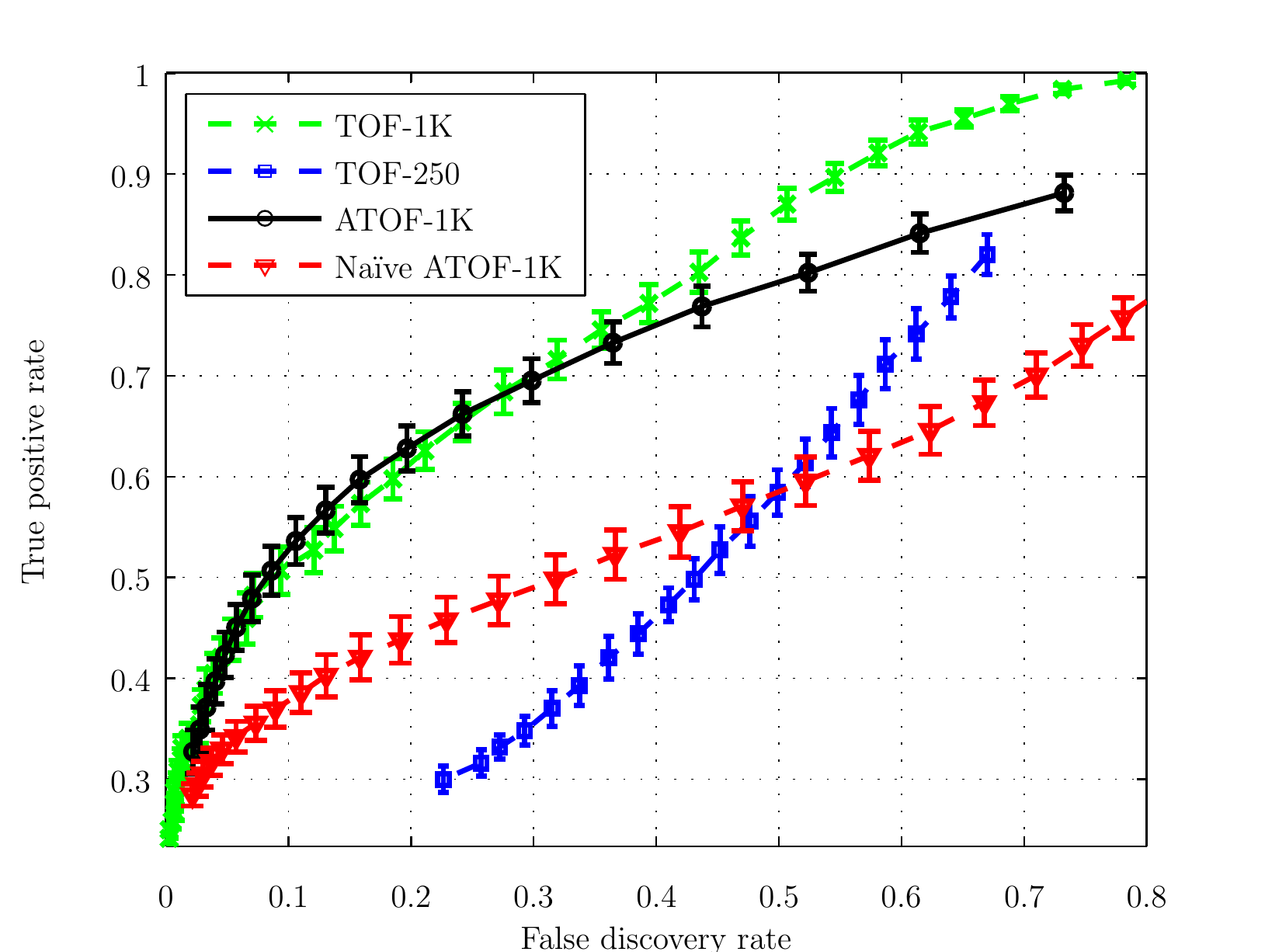}
\caption{True positive rate vs false discovery rate}
\label{fig:tpr_fdr_theta}
\end{figure}

The goal of a TOFMS instrument is to precisely measure the MCR of the present ions. 
Hence, for the instrument to be able to perform accurate peak picking it is important that \newtof{} does not significantly distort the shape of the pulses at the estimated spectrum.
Figure~\ref{fig:width_intesity_ratio} shows, for \newtof{} and \ctof{} and for different number of scans and acceleration factor ten, the empirical CDF of the width to intensity (hight) ratio of the peaks (a measure of peak spikiness).
As can be seen \newtof{} introduces little broadening in the peaks.

To better understand the achieved accuracy, we compare \newtof{} and three other cases in Fig.~\ref{fig:tpr_fdr_theta}.
We run multiple randomized sample reconstruction and calculate bootstrap confidence intervals.
The figure shows the true positive rate vs. false discovery rate, as
parameterized by $\theta_0$ for the estimated spectrum $\widehat{\boldx}_{\newtof{}}$ obtained by the \newtof{} with one thousand scans and acceleration factor 4.
Each data point is obtained dividing the ten thousand scans into ten buckets of one thousand scans each.
The trace is constructed using one thousand scans from one bucket.
The average of the other nine buckets is considered as the ground truth. 
Error bars indicate $2\hat{\sigma}$ confidence intervals where $\hat{\sigma}$ is the standard error estimate.  

We Also plot  the corresponding curves for three other cases.
The red curve corresponds to the \emph{na\"{\i}ve \newtof}, obtained by mapping each event to all possible positions on the spectrum,
\begin{equation}
\widehat{\boldx}_{N}[i] = \sum_{t \in F}\frac{1}{\textrm{deg}_t}A_{tj} \boldy[t],
\end{equation}
where $A_{tj}$ is the adjacency matrix defined in
Eq.~\eqref{eq:adjacency_matrix} and $\textrm{deg}_t = \sum_{j=1}^n A_{tj}$ is the number of positions on the spectrum that event $t$ can be mapped to.
The na\"{\i}ve \newtof uses the simplest way of processing the \newtof trace where one assumes that each event is equally likely to be caused by an ion from any of its potential locations on the spectrum.

The green (TOF-1K) and blue (TOF-250) curves correspond to  conventional TOFMS, i.e., the spectrum obtained by averaging the scans when there is no overlapping,
\begin{equation}
\widehat{\boldx}_{\rm ave} = \frac{1}{N} \sum_{l=1}^{N} \boldx^{(l)}.
\end{equation}
In TOF-250, the number of scans is chosen such that the time which takes to perform the TOFMS and \newtof{} are the same.
TOF-1K corresponds to TOFMS with the same number of scans as \newtof, which has an acquisition time $4$ times longer compared to the \newtof{}.
Equivalently, one can think of TOF-1K corresponding to the case where an oracle is available for \newtof{} experiment that to each event assigns its true position on the spectrum. 

The na\"{\i}ve \newtof and conventional TOFMS curves are parameterized by the thresholding parameter $h_w$ which is used to identify the events in the estimated spectrum.
Similar to $\theta_0$, $h_w$ enables us to obtain a trade off between true positive rate and false discovery rate.  

This comparison shows that \newtof{} significantly outperforms the conventional TOFMS.
One the one hand, it allows for a four-fold speed-up with essentially unchanged accuracy (comparison with TOF-1K).
On the other, it allows a two-fold increase in true positive rate for ${\rm FDR}= 0.2$ if the experiment duration is unchanged (comparison with TOF-250).


\section{Events that span more than one sample}\label{sec:generalization}
In order to simplify the presentation of the main ideas in the previous sections we assumed that all ion impacts are confined to one sample in the observed trace. 
However, as was evident from Fig.~\ref{fig:sample_event} this assumption is far from being true. 
In this section we show how the algorithm presented in Section \ref{sec:algorithm} can be extended to the case where each ion impact event can spread across multiple samples in the trace. 

Consider the conditional probability of an observation $\boldy[t]$ from Eq.~\eqref{eq:density} which is repeated below for convenience. 
\begin{align}\label{eq:density2}
\bP(\boldy[a]|\boldw, A) &= e^{-\langle A_a, \boldw\rangle - w^0} \delta(\boldy[a]) \nonumber\\
& \quad + \sum_{k=1}^{\infty}{E_{k,\mu}(\boldy[a]) P_{\langle A_a, \boldw\rangle + w^0}(k) }.
\end{align}
%

Assume instead of being confined to one sample the ion impact event $a$ is spread from time $\ut_a$ to $\ot_a$. 
Define the weight of ion impact $a$, $z_a$, as
\begin{equation}\label{eq:impact_weight}
z_a = \sum_{t = \ut_a}^{\ot_a} \boldy[t].
\end{equation}
Recall that $\boldw[i] + w^0$ represents the arrival rate of ions at bin $i$ for a single scan, i.e., $\boldw[i] + w^0$ is the expected number of ions that arrive at bin $i$ in a single scan.
Henceforth, assuming that ions arrive independent of each other $\langle A_t, \boldw + w^0 \mathbf{1}\rangle$ is the cumulative mean of the number of ions that arrive at time $t$ where $\mathbf{1}$ is the vector of all ones.
Given the assumption that an ion impact being confined to only one sample we used $\langle A_t, \boldw + w^0 \mathbf{1}\rangle$ as the expected number of ions that caused the observation $\boldy[t]$.
Here, we make the alternative assumption that for an ion impact observation spanning the interval $[\ut_a, \ot_a]$ the expected number of ions responsible, $K$, is given by 
\begin{equation}\label{eq:cumulative_rate}
  \expect[K] = \sum_{t = \ut_a}^{\ot_a} \langle A_t, \boldw + w^0\rangle
\end{equation}
Let the generalized neighborhood of an event be
\begin{align}\label{eq: gen_neighbor_def}
  \partial a = \{i \in [n] \,|\, \exists \, t \in [\ut_a, \ot_a] \; \text{s.t.} \; A_{ti} > 0 \}.
\end{align}
Then, Eq.~\ref{eq:cumulative_rate} can be written as 
\begin{equation}\label{eq:cumulative_rate_2}
  \expect[K] = \sum_{i \in \partial a}(\boldw[i] + w^0).
\end{equation}
In other words, we are assuming that an observed event can be caused by one or multiple ions impacting at any time during the span of the event.

Similar to Section \ref{sec:data_model} we assume that $z$, the weight of an ion impact, is distributed as an Erlang random variable with shape parameter $K$ where $K$ is a Poisson random variable with mean given in Eq. \eqref{eq:cumulative_rate}.
Let $a$ be an ion impact event with weight $z_a$ and expected number of ions $\sum_{i \in \partial a}\boldw[i]$. Then, similar to Eq. \eqref{eq:density} the probability of observing $a$ given $\boldw$ and $A$ can be written as
\begin{align}
  &\bP(a|\boldw, A) = \exp\left\{-\sum_{i \in \partial a}(\boldw[i] +w^0) \right\} \delta(z_a) \nonumber\\
& \quad + \exp\left\{-\frac{z_a}{\mu}-\sum_{i \in \partial a}(\boldw[i] + w^0)\right\} \, \nonumber \\
& \qquad \qquad \quad \sum_{k=1}^{\infty}{\frac{z_a^{k-1}}{\mu^{k}(k-1)\! \,!k\!\,!} \, (\sum_{i \in \partial a}\boldw[i] + w^0 )^{k} }.
\label{eq:density_impact_weight}
\end{align}

For a given trace $\boldy$, let $F$ be the set of observed events.
Similar to section \ref{sec:data_model} assume that given $\boldw$ and $A$ distinct events (ion impacts) are independent.
Then the probability of observing the set of non-zero-weight events $F$ can be written as
\begin{align}\label{eq:density_impact_weight_joint}
\bP(F|\boldw, A) &= \prod_{a \in F} \bigg[ \exp\left\{-\frac{z_a}{\mu}-\sum_{i \in \partial a} (\boldw[i]+w^0) \right\} \nonumber \\
& \qquad \qquad \sum_{k=1}^{\infty}{\frac{z_a^{k-1} }{\mu^{k} {(k-1)\!\,!} {k\!\,!} } \, \, (\sum_{i \in \partial a}\boldw[i]+w^0)^{k} } \bigg].
\end{align}
Note that here we dropped the first term in Eq \eqref{eq:density_impact_weight} since for an observed ion impact event the weight $z_a$ is always positive and the first term vanishes. 

The next step is to write the negative log-likelihood function. 
There is a subtle point to be noted here. 
$F$ is the set of ion impact events with non-zero weights.
However, $z_a$ can be zero while $\sum_{i \in \partial a} \boldw[i] > 0$. 
We refer to such observations as zero-weight observations.
Zero-weight observations are informative and should be included in the log-likelihood function.
In our model the conditional probability of observing a zero-weight event $a$ in the interval $[\ut_a, \ot_a]$ is given by
\begin{align}\label{eq:density_zero_weight}
  \bP(a|\boldw, A) &= \exp\left\{- \sum_{i \in \partial a} (\boldw[i]+w^0) \right\} \nonumber \\
  & = \exp\left\{- \sum_{t = \ut_a}^{\ot_a} \langle A_t, (\boldw+w^0)\rangle \right\}.
\end{align}
The difficulty that seems to exist here is how to identify the zero-weight ion impact events when they can span more than one sample. 
However, as we shall see shortly, the particular form of the log-likelihood function for the zero-weight events eliminates the need to distinguish between adjacent zero-weight events for calculating the log-likelihood function. 

Let $F_0$ be the set of zero-weight observations and define $U_0 = \cup_{a \in F_0} [\ut_a, \ot_a]$, i.e., the union of all time intervals we observed zero-weight events at.
Again, making the assumption that distinct zero-weight ion impacts are independent events given $\boldw$ and $A$ we can write the joint probability of observing $F_0$ as
\begin{align}\label{eq:density_zero_weight_joint}
\bP(F_0|\boldw, A) &= \prod_{a \in F_0} \exp\left\{- \sum_{i \in \partial a} \boldw[i] \right\} \nonumber \\
&= \prod_{a \in F_0} \exp\left\{- \sum_{t = \ut_a}^{\ot_a} \langle A_t, \boldw\rangle \right\} \nonumber \\
&= \exp(- \sum_{t \in U_0} \langle A_t, \boldw\rangle).
\end{align}
Note that $\bP(F_0|\boldw, A)$ does not depend on the number of zero-weight events or the beginning or end time of a particular event. We need only to identify all the samples that are part of a zero-weight event, i.e., not part of any observed ion impact event. 

Putting Equations \eqref{eq:density_impact_weight_joint} and \eqref{eq:density_zero_weight_joint} together and assuming that all the events are independent we have the probability of observing a trace $\boldy$ as
\begin{align}\label{eq:generalized_prob_trace}
 & \bP(\boldy|\boldw, A) = \bP(F|\boldw, A)  \bP(F_0|\boldw, A)  \nonumber \\
 & \quad = \exp\left\{- \sum_{a \in F_0} \sum_{i \in \partial a} (\boldw[i] + w^0) \right\} \nonumber \\
& \qquad \exp\left\{ - \sum_{a \in F}\sum_{i \in \partial a}(\boldw[i]+w^0) - \frac{z_a}{\mu} \right\} \nonumber \\
& \qquad \prod_{a \in F} \bigg[ \sum_{k=1}^{\infty}{\frac{z_a^{k-1} }{\mu^{k} {(k-1)\!\,!} {k\!\,!} } \, \, \left(\sum_{i \in \partial a}(\boldw[i] + w^0)\right)^{k} } \bigg] \nonumber \\
& \quad = \exp\left\{- N \|\boldw\|_1 -\sum_{a \in F} \frac{z_a}{\mu} \right\} \nonumber \\
& \qquad\qquad  \prod_{a \in F} \left[ \, \sum_{k=1}^{\infty}{\frac{z_a^{k-1} }{\mu^{k} {(k-1)\!\,!} {k\!\,!} } \, \left(\sum_{i \in \partial a}(\boldw[i] + w^0)\right)^{k} } \right],
\end{align}
where the last equality is obtained since each sample on the trace is presented in either $F$ or $F_0$. 
The negative log-likelihood function then can be written as
\begin{align}\label{eq:generalized_neg_log_likelihood}
& \ell(\boldw| \boldy, A) = - N \|\boldw\|_1 \nonumber \\
& \qquad - \sum_{a \in F} \log \left[ \sum_{k=1}^{\infty}{\frac{z_a^{k} }{\mu^{k} {(k-1)\!\,!} {k\!\,!} } \, \left( \sum_{i \in \partial a}(\boldw[i] + w^0)\right)^{k}  } \right]\end{align}
After some algebra, and keeping only the terms that depend on $\boldw$,
\begin{align}
 \ell(\boldw| \boldy, A) = & - N \|\boldw\|_1  \nonumber \\
& - \sum_{a \in F} \bigg[ \frac{1}{2} \log\left(\sum_{i \in \partial a}\frac{(\boldw[i] + w^0)}{\mu} \right) \nonumber \\
&+ \log  \bessel_1 \left( 2 \sqrt{z_a \sum_{i \in \partial a} \frac{(\boldw[i] + w^0)}{\mu }} \right)  \bigg],
\end{align}
Using the same change of variable as before, namely $\boldtw = \frac{1}{\mu}\boldw$, $\lambda = N \mu$ and $\tell(\boldtw| \boldy, A) = \ell(\boldw| \boldy, A) - N \|\boldw\|_1$ we have $\ell(\boldw| \boldy, A) = \tell(\boldtw| \boldy, A) + \lambda \|\boldtw\|_1$ where
\begin{align}
 \tell(\boldtw| \boldy, A) = & - \sum_{a \in F} \bigg[ \frac{1}{2} \log\left(\sum_{i \in \partial a}(\boldtw[i] + w^0) \right) \nonumber \\
&+ \log  \bessel_1 \left( 2 \sqrt{z_a \sum_{i \in \partial a} (\boldtw[i] + w^0)} \right)  \bigg],
\end{align}
And the gradient of the function $ \tell(\boldw| \boldy, A) $ is 
\begin{align}\label{eq:generalized_neg_log_likelihood_gradient}
&\nabla \tell(\boldw| \boldy, A) =  - \sum_{a \in F} \sum_{t = \ut_a}^{\ot_a} A_t \Bigg[ \frac{1}{2\sum_{i \in \partial a}(\boldw[i] + w^0)} \nonumber \\
& \quad + \left(\bessel_0 ( 2 \sqrt{z_a \sum_{i \in \partial a}(\boldw[i] + w^0) } ) + \bessel_2 ( 2 \sqrt{z_a\sum_{i \in \partial a}(\boldw[i] + w^0) } )\right) \nonumber \\
& \quad \left(2 \sqrt{z_a \sum_{i \in \partial a}(\boldw[i] + w^0)} \bessel_1 ( 2 \sqrt{z_a \sum_{i \in \partial a}(\boldw[i] + w^0) } )\right)^{-1} \Bigg]  
\end{align}

Having the gradient of the log-likelihood function the algorithm is similar to the algorithm of section \ref{sec:algorithm} with the gradient of the log-likelihood calculated using Eq. \eqref{eq:generalized_neg_log_likelihood_gradient}.
However, we need some additional notations to represent the generalized algorithm.
Let $\deg(a)$ be the number of neighbors of event $a$, i.e., $\deg(a) = \sum_{i = 1}^{n} A_{\ut_a,i}$ \footnote{This definition is slightly inaccurate since an event can potentially fall on the boundary of an scan resulting in $\sum_{i = 1}^{n} A_{\ut_a,i} \neq \sum_{i = 1}^{n} A_{\ot_a,i}$ but this is rare and the discrepancy is negligible.}.
For $j \in [\deg(a)]$ let $\ui_a^j$ be the index of the $j^{\rm th}$ non-zero element of $A_{\ut_a}$. 
Similarly, let $\oi_a^j$ be the index of the $j^{\rm th}$ non-zero element of $A_{\ot_a}$. 
Then, $[\ui_a^j, \oi_a^j]$ is the true position of event $a$ on the spectrum if its $j^{th}$ neighbor corresponds to the true scan that caused event $a$.

The peak picking algorithms are usually sensitive to the shape of the pulses.
Furthermore, the time of arrival of the ions are noisier than the observation error of the instrument.
Observing many scans enables the instrument to measure the MCR of the ions with precision significantly better than the arrival noise level. 
To overcome the issue of a possible bias in the estimated MCR in our model, we employ one last trick.
The algorithm constructs an estimate of the spectrum $\widehat{\boldx}$ by assigning each observed event to its most likely neighbor (c.f. Fig.~\ref{fig:bipartite_graph}).
In other words, let
\begin{equation}
     j_* = \underset{j \in [\deg(a)]}{\arg\max} \; \sum_{i = \ui_a^j}^{\oi_a^j} \boldw[i].
\end{equation}
Then, given $\boldw$ we reconstruct the estimated spectrum as
\begin{equation}
     \widehat{\boldx}[\ui_a^{j^*} + \Delta] = \widehat{\boldx}[\ui_a^{j^*} + \Delta] + \boldy[t + \Delta].
\end{equation}
Using this notation the algorithm as as follows.

\begin{center}
     \begin{tabular}[c]{l}
     \hline
     \textbf{Algorithm} \\
     \hline
     \textbf{Input:} trace $\boldy$, firing times $\boldtau$, and constants ($\theta_0$, $\theta_1$, $\lambda$, $\gamma$)\\
     \textbf{Output:} estimated spectrum $\widehat{\boldx}$\\
     1:\quad Calculate the adjacency matrix $A$ as in Eq.~\eqref{eq:adjacency_matrix}.\\
     2:\quad  Set $\boldw^{(0)} = 0$, $\theta = \theta_0 + \theta_1$. \\
     2:\quad Repeat until stopping criterion is met:\\
       \quad \qquad $\theta \leftarrow \theta_0 + \frac{1}{k^2} \theta_1$\\
       \quad \qquad $\boldw^{(k+1)} \leftarrow \eta_{\theta}\left( \boldw^{(k)} - \gamma \nabla \tell(\boldw^{(k)}; \, \boldy, A) \right).$ \\
     3:\quad Set $\widehat{\boldx} = 0$\\
     4:\quad For $a \in F$:\\
     \quad \qquad $j^* = \underset{j \in [\deg(a)]}{\arg\max} \; \sum_{i = \ui_a^j}^{\oi_a^j} \boldw[i]$ \\
     \qquad \qquad For $\Delta = 0, \dots, (\ot_a - \ut_a)$:\\
     \qquad \quad \qquad $\widehat{\boldx}[\ui_a^{j^*} + \Delta] = \widehat{\boldx}[\ui_a^{j^*} + \Delta] + \boldy[t + \Delta]$\\
     5:\quad Return $\widehat{\boldx}$.\\
      \hline
      \end{tabular}
\end{center}
%

\bibliographystyle{IEEEbib}
\bibliography{lad_references}

\clearpage 


\begin{IEEEbiography}{Morteza Ibrahimi}
Morteza Ibrahimi received his PhD in Electrical Engineering from Stanford University in 2013,
working with professor Andrea Montanari. His research interests are in high
dimensional statistical signal processing, learning graphical models, optimization through
message passing algorithms, and statistical inference with high dimensional data, specially
through fast iterative algorithms. He received his BS from Sharif University of Technology
in 2006, and MSc from University of Toronto in 2007.
\end{IEEEbiography}

\begin{IEEEbiography}{Andrea Montanari}
Andrea Montanari received a Laurea degree in Physics in 1997, and a Ph. D. in Theoretical Physics in 2001 (both from Scuola Normale Superiore in Pisa, Italy). He has been post-doctoral fellow at Laboratoire de Physique Théorique de l'Ecole Normale Supérieure (LPTENS), Paris, France, and the Mathematical Sciences Research Institute, Berkeley, USA. Since 2002 he is Chargé de Recherche (with Centre National de la Recherche Scientifique, CNRS) at LPTENS. In September 2006 he joined Stanford University as a faculty, and since 2010 he is Associate Professor in the Departments of Electrical Engineering and Statistics.
He was co-awarded the ACM SIGMETRICS best paper award in 2008. He received the CNRS bronze medal for theoretical physics in 2006 and the National Science Foundation CAREER award in 2008.
\end{IEEEbiography}

\begin{IEEEbiography}{George Moore}
Goerge S. Moore is a research fellow at Agilent technology. He received his PhD 
in 1980 from Purdue University and his Bsc/Msc from Mississippi State University in 1974 all in
Electrical Engineering. He has over 40 years of experience in the industry.
\end{IEEEbiography}

\end{document}